\documentclass{amsart}

\usepackage{mathrsfs}
\usepackage[all]{xy}

\usepackage{wasysym}
\usepackage{amssymb}
\usepackage{MnSymbol}
\usepackage{comment}

\usepackage{mathtools}

\usepackage{ifpdf}
\ifpdf 
  \usepackage[pdftex]{graphicx}
  \DeclareGraphicsExtensions{.pdf,.png,.jpg,.jpeg,.mps}
  \usepackage{pgf}
\else 
  \usepackage{graphicx}
  \DeclareGraphicsExtensions{.eps,.bmp}
  \usepackage{pgf}
  \usepackage{pstricks}
\fi
\usepackage{epic,bez123}
\usepackage{wrapfig}

\newtheorem{thm}{Theorem}[section]
\newtheorem{prop}[thm]{Proposition}
\newtheorem{cor}[thm]{Corollary}
\newtheorem{lem}[thm]{Lemma}
\newtheorem*{sublem}{subLemma}

\newtheorem{const}[thm]{Constants}

\newtheorem*{claim}{Claim}
\newtheorem*{claim1}{Claim 1}
\newtheorem*{claim2}{Claim 2}

\newtheorem{conv}[thm]{Convention}

\theoremstyle{definition}
\newtheorem{defn}[thm]{Definition}

\theoremstyle{remark}

\newtheorem{obs}[thm]{Observation}

\newtheorem*{rem}{Remark}

\newtheorem*{ack}{Acknowledgment}

\newcommand{\Gx }{\mathscr{G} (G, S)}

\newcommand{\GP }{(G, \mathcal P)}

\newcommand{\act}{\curvearrowright}

\newcommand{\pGf}{\partial_\lambda{G}}

\newcommand{\pG}{\Lambda G}

\newcommand{\pX}{\partial X}

\newcommand{\PS}[1]{\mathcal P_{#1}(s,o)}
\newcommand{\Ps}[1]{\mathcal P_{#1}(s,1)}

\newcommand{\g}[1]{\delta_{#1}}

\newcommand{\diam }[1]{{\textbf{diam}(#1)}}

\newcommand{\proj }{\mbox{Pr}}

\newcommand{\dirac}[1]{{\mbox{Dirac}}{(#1)}}

\newcommand{\len }{\ell}

\usepackage[bookmarks=true, pdfauthor={YANG Wenyuan}]{hyperref}

\title[Purely exponential growth of cusp-uniform actions]{Purely exponential growth of cusp-uniform actions}

\author{Wen-yuan Yang}

\address{Beijing International Center for Mathematical Research \&
School of Mathematical Sciences, Peking University, Beijing, 100871,
P.R.China}

\email{yabziz@gmail.com}

\begin{document}


\date{}

\dedicatory{}


\begin{abstract}
Suppose that a countable group $G$ admits a cusp-uniform action on a hyperbolic space $(X,d)$ such that $G$ is of divergent type. The main result of the paper is characterizing the purely exponential growth type of the orbit growth function by  a condition introduced by Dal'bo-Otal-Peign\'e.  For geometrically finite Cartan-Hadamard manifolds with  pinched negative curvature this condition ensures the finiteness of Bowen-Margulis-Sullivan measures. In this case, our result recovers a theorem of Roblin (in a coarse form). Our main tool is the Patterson-Sullivan measures on the Gromov boundary of $X$, and a variant of the Sullivan shadow lemma called partial shadow lemma. This allows us to prove that the purely exponential growth of either cones, or partial cones or horoballs is also equivalent to the condition of Dal'bo-Otal-Peign\'e. These results are further used in the paper \cite{YANG7}.
\end{abstract}

\maketitle

\setcounter{tocdepth}{1} \tableofcontents

\section{Introduction}\label{Section1}

Suppose that a group $G$ admits a proper and isometric action on a proper geodesic
hyperbolic space $(X, d)$ such that $G$ does not fix a point in the Gromov boundary $\pX$. We are interested in studying the asymptotic feature of the $G$-orbits in $X$ via the Patterson-Sullivan measure on the boundary $\pX$. This is a recurring scheme in the setting of simply connected Cartan-Hadamard manifolds with pinched negative curvature. The novelty of the present paper is the generality of the results obtained for a class of cusp-uniform actions we explain now.

The \textit{limit set} $\Lambda (G)$ of $G$ is the set of accumulation points in $\pX$ of a $G$-orbit in $X$. 
It is well-known that $G$ acts minimally on $\Lambda (G)$ as a convergence group action. A point $p\in \Lambda(G)$ is called \textit{parabolic} if the stabilizer $G_p=\{g\in G: gp=p\}$ is infinite so that its limit set is $\{p\}$. Denote by $\mathcal P$ the set of
maximal parabolic subgroups in $G$. Let $\mathcal C(\Lambda G)$ be the \textit{convex-hull} of $\Lambda(G)$ that is the union of all bi-infinite geodesics with both endpoints in $\Lambda(G)$. It is known that $\mathcal C(\Lambda G)$ is a $G$-invariant quasiconvex subset. So $\mathcal C(\Lambda G)$ is itself a hyperbolic space.  

We consider the following definition stated in \cite{Hru}, generalizing the definition of a
geometrically finite Kleinian group in \cite{Marden}. 
\begin{defn}[Cusp-uniform action]\label{RHdefn}
Assume that there is a
$G$-invariant system of (open) horoballs $\mathbb U$ centered at
parabolic points of $G$ such that the action of $G$ on the complement
$$\mathcal C(\Lambda G) \setminus \mathcal U,\; \mathcal U := \cup_{U \in \mathbb U} U$$
is co-compact. Then the pair $\GP$ is said to be \textit{relatively
hyperbolic}, and the action of $G$ on $X$ is called
\textit{cusp-uniform}.
\end{defn}
By a compactness argument, we see that $\mathbb U$ contains only finitely many $G$-orbits. Consequently, the set $\mathcal P$ of maximal parabolic subgroups has only finitely many conjugacy classes. See \cite{Bow1} and \cite{Tukia} for more details.
\begin{rem}
\begin{enumerate}
\item
We emphasize that $G$ is not required to be finitely generated: since $\pX$ is metrizable, the group $G$ could be at most countable \cite[Corollary 7.1]{Ge1}. The free product of two non-finitely generated countable groups gives rise to a non-finitely generated cusp-uniform action.  
\item 
We do not request that $\pX=\Lambda G$ and consider using the convex hull of $\Lambda G$ instead. This allows us to include  examples of a geometrically finite Kleinian group acting on $\mathbb H^n$ $(n\ge 2)$ with $\Lambda G \subsetneq \mathbb S^{n-1}$. 
\end{enumerate}
\end{rem}

In fact, a group (pair) being relatively hyperbolic  
admits many equivalent formulations, for instance \cite{Farb}, \cite{Bow1},
\cite{DruSapir}, \cite{Osin}, \cite{Hru} and \cite{Ge1}.  Since we are interested in the asymptotic growth  of the orbits of a cusp-uniform action, the notion of a critical exponent shall be our focus. 

Choose a basepoint $o \in X$. Set $N(o, n):=\{g\in G: d(o, go)\le n\}$. Consider the Poincar\'e series for a subset $\Gamma \subset G$:
$$\PS{\Gamma} = \sum\limits_{g \in \Gamma} \exp(-sd(o, go)), \; s \ge 0.$$
Note that the \textit{critical exponent} of $\PS{G}$ is given by
$$\g \Gamma = \limsup\limits_{n \to \infty} \frac{\log \sharp(N(o, n)\cap \Gamma)}{n},$$
which is independent of the choice of $o \in X$. Observe that $\PS{\Gamma}$ diverges for $s<\g \Gamma$ and converges for $s>\g \Gamma$.
 
The group $G$ is of \textit{divergent type} (resp.
\textit{convergent type}) (with respect to the action of $G$ on $X$) if $\PS{G}$ is divergent (resp. convergent) at
$s=\g G$. Note that whether $\PS{G}$ is of divergent type does not depend on the choice of $o$.

\subsection{Patterson-Sullivan measures and the partial shadow lemma}  
The primary tool in the paper is the Patterson-Sullivan measures (PS-measures for shorthand) on the Gromov boundary of $X$. For discrete groups acting on $n$-dimensional hyperbolic spaces $(n\ge 2)$ the theory of PS-measures was established by Patterson \cite{Patt} for $n=2$ and generalized by Sullivan
\cite{Sul} in all dimensions. Furthermore, Sullivan gave a way to construct a flow-invariant measure on the unit tangent bundle for hyperbolic $n$-manifolds, which coincides with the Bowen-Margulis measures studied earlier. Sullivan's construction is very robust and applies in rather general settings, for instance in CAT(-1) spaces in \cite{Roblin}. Following  \cite{Roblin} we call this measure Bowen-Margulis-Sullivan measure  (BMS measure for short). If $X=\mathbb H^n$ ($n\ge 2$) is a real hyperbolic space, then the BMS measure is finite \cite[Theorem 3]{Sul2}. However, for Cartan-Hadamard manifolds with pinched negative curvature,  it was observed in \cite{Yue} that the finiteness of the BMS measure depends crucially on the geometry on cusps.

In what follows, we shall discuss in details three conditions with increasing generalities on parabolic subgroups.
 
By definition, a cusp-uniform action of $G$ on $X$ has a \textit{parabolic gap property} (PGP) if $\g
G > \g P$ for every maximal parabolic subgroup $P$. This property was introduced by Dal'bo-Otal-Peign\'e in \cite{DOP} to deduce  that $G$ is of divergent type. They also introduced in the same paper another condition which ensures the finiteness of the BMS measure for Cartan-Hadamard manifolds  with pinched negative curvature.  We formulate their condition in the setting of cusp-uniform actions.

\begin{defn}
Let $G$ admit a cusp-uniform action as above such that $\g G < \infty$. Then $G$ satisfies  a \textit{Dal'bo-Otal-Peign\'e (DOP) condition} if the following 
\begin{equation}\label{BMS}
\sum_{p\in P} d(o, po) \exp(-\g G d(o, po)) < \infty
\end{equation}
holds for every $P\in \mathcal P$ and $o\in X$.
\end{defn}
\begin{rem}
The condition (\ref{BMS})  depends only on conjugacy classes of $P\in \mathcal P$. In practice, it suffices to verify this condition for finitely many conjugacy classes in $\mathcal P$. Note also that (\ref{BMS}) does not depend on the choice of $o\in X$. 
\end{rem}

\begin{thm}\cite{DOP}\label{DOPthm}
Let $G$ admit a cusp-uniform action on a simply connected Cartan-Hadamard manifold $X$ with pinched negative curvature. Suppose that $G$ is of divergent type. Then the BMS measure is finite on the unit tangent bundle of $X/G$ if and only if $G$ satisfies the DOP condition. 
\end{thm}

It is obvious that if $G$ has the parabolic gap property then $G$ satisfies the DOP condition. They imply  the third one - \textit{a parabolic convergence property} (PCP): 
\begin{equation}\label{PCP}
\sum_{p\in P} \exp(-\g G d(o, po)) < \infty
\end{equation}
for every $P\in \mathcal P$ and $o\in X$. This property turns out to be always true: it is trivial for the convergent case; see Lemma \ref{convpara} for the divergent case. Moreover, the PCP property is one of crucial facts used to establish the partial shadow lemma below. 

On the other hand, there exist  examples of cusp-uniform actions of divergent type but without DOP condition, and satisfying the DOP condition but without PGP property. See \cite{Peigne} for these examples. The above three conditions  on parabolic groups will be important in further discussions. 

Generalizing Patterson and Sullivan's work, Coornaert  \cite{Coor} has established the theory of PS-measures
on the limit set of a discrete group acting on a $\delta$-hyperbolic
space. Assuming that $G$ is of divergent type, we first generalize Dal'bo-Otal-Peign\'e's results to the setting of cusp-uniform actions.
  
\begin{thm}[=Proposition \ref{PSMeasureX}]\label{ThmP} 
Let $G$ admit a cusp-uniform action on $X$ such that $\g G < \infty$ and $G$ is of divergent type.
Then the Patterson-Sullivan measure $\{\mu_v\}_{v\in G}$ on $\pX$ is a quasi-conformal
density without atoms. Moreover, $\{\mu_v\}_{v\in G}$ is unique and ergodic.
\end{thm}
\begin{rem}
The statement that $\{\mu_v\}_{v\in G}$ on $\pX$ is a quasi-conformal
density was proved by Coornaert \cite{Coor}. The new point here are the ``without atoms" and ``moreover" statements.
\end{rem}

In the theory of PS-measures, a key tool is the Sullivan Shadow
Lemma, which connects the geometry inside and the measure on
boundary.   We shall prove a variant of
the Shadow Lemma that holds for \textit{partial shadows}, which excludes a ``small'' set of points shadowed by a system of quasiconvex subsets. To make this idea precise, we need to introduce a technical definition.

Let $\mathbb U\subset \mathbb Y$ be a $G$-finite system of quasiconvex subsets with bounded intersection such that, for each $Y\in \mathbb Y\setminus \mathbb U$, the stabilizer $G_Y$ acts co-compactly on $Y$.  The notion of transition points was due to Hruska \cite{Hru} in the setting of Cayley graphs.  Here it is formulated in a general metric setting. 

\begin{defn}
For a path $p$ in $X$, a point $v \in p$ is called
an \textit{$(\epsilon, R)$-transition} point for $\epsilon, R \ge 0$
if the $R$-neighborhood $\displaystyle{p \cap B(v, R)}$ around $v$ in $p$ is not
contained in the $\epsilon$-neighborhood of any $Y \in \mathbb Y$.
\end{defn}

Hence, the \textit{partial shadow} $\Pi_{r,\epsilon, R}(go)$ at $g$
for $r \ge 0$ is the set of boundary points $\xi \in \pG$ such that
some geodesic $[o, \xi]$ intersects $B(go, r)$ and
contains an $(\epsilon, R)$-transition point $v$ in $B(go, 2R)$. Thanks to parabolic convergence property  and no atoms at parabolic points, the partial shadow lemma claims that the excluded points are negligible when $R$ is sufficiently large.  

\begin{lem}[=Lemma \ref{PShadowX}]\label{PSLem}
Under the assumption of Theorem \ref{ThmP}, there are constants $r_0,
\epsilon, R\ge 0$ such that the following holds
$$
\exp(-\g G d(o, go)) \prec \mu_1(\Pi_{r, \epsilon, R}(go)) \prec_r
\exp(-\g G d(o, go)) ,
$$
for any $g \in G$ and $r \ge r_0$.
\end{lem}

We have pointed out that the PCP property holds automatically if $G$ is of convergent type. So the other essential ingredient is due to the divergence of $G$ implying that parabolic points (and  the limit set of $G_Y$ as above) are not charged at all by PS-measures (cf. Lemma \ref{parabnoatom2}).

We conclude  by giving some comments on the above partial things and the choice of a system $\mathbb Y$.  Usually, the properties of parabolic subgroups  play  a decisive role in establishing negatively-curved aspects of relatively hyperbolic groups. So it is common practice to take the effect of them in control to secure an analogous theory of hyperbolic groups.  In this regard, the notion of partial shadows appear to be a natural generalization of, and a useful complement to normal shadows. Furthermore, the choice of a strictly larger system $\mathbb U\subset \mathbb Y$ is essential feature being made use of in \cite{YANG7}. So in next subsection, we shall indeed deal with partial cones with respect to a general system $\mathbb U\subset \mathbb Y$ with applications towards \cite{YANG7} and future development.

\subsection{Characterizing   purely exponential type of growths}
This is the main contribution of this paper. The idea of using PS-measures to study growth
problems goes back to works of Patterson \cite{Patt} and Sullivan
\cite{Sul}, see \cite[Theorem 9]{Sul} for example. We now introduce several growth functions associated to orbits and horoballs before stating the main result.

For $\Delta \ge 0, n \ge 0$, consider the orbit in an annulus 
\begin{equation}\label{OrbitFunction}
A(go, n, \Delta) := \{h \in G:  n-\Delta \le d(o, ho) - d(o, go) <
n+\Delta\},
\end{equation}
for any $g \in G$.  

\paragraph{\textbf{Partial cones}} Corresponding to a partial shadow, a notion of a partial cone could be similarly defined with respect to a system $\mathbb Y$ of quasiconvex subsets with bounded intersection. 

The \textit{$r$-cone} $\Omega_r(go)$ at $go$ for $r\ge 0$ is the
set of elements $h \in G$ such that some geodesic $[o, ho]$
intersects $B(go, r)$. For $r, \epsilon, R > 0$, the notion of an \textit{$(\epsilon, R)$-partial $r$-cone}
$\Omega_{r, \epsilon, R}(go)$ at $go$ is defined similarly, by
demanding the existence of $(\epsilon, R)$-transition points on $[o,
ho]$ $2R$-close to $go$. See Section \ref{Section3} for precise definitions.

Consider the orbit in a cone:
\begin{equation}\label{ConeOrbitFunction}
\Omega_r(go, n, \Delta) := \Omega_r(go) \cap A(go, n, \Delta)
\end{equation}
and in a partial cone:
\begin{equation}\label{PConeOrbitFunction}
\Omega_{r, \epsilon, R}(go, n, \Delta) := \Omega_{r, \epsilon, R}(go)
\cap A(go, n, \Delta),
\end{equation}
for any $g \in G, n \ge 0$. 

\paragraph{\textbf{Purely Exponential growth}}
We say that the orbit growth of $G$ is of \textit{purely exponential type} if there exists $\Delta>0$ such that 
\begin{equation}\label{PExpGrowth}
\sharp A(o, n, \Delta) \asymp \exp(\g G n)
\end{equation}
for $n\ge 1$. If there exist $r, \epsilon, R, \Delta>0$ such that (\ref{PExpGrowth}) holds for $\sharp \Omega_r(go, n, \Delta)$ (resp. $\sharp\Omega_{r, \epsilon, R}(go, n, \Delta)$) then the orbit growth in cones (resp. partial cones) of $G$ is of \textit{purely exponential type}.

The main result of the paper is that the above growth functions of purely exponential type are all equivalent to the DOP condition. They are in fact also equivalent to the purely exponential growth of horoballs defined as follows.

Let $\mathbb U$ be the collection of horoballs in   definition \ref{RHdefn} of a cusp-uniform action. Consider 
\begin{equation}\label{HExpGrowth}
H(o, n, \Delta):=\{U\in \mathbb U:   \Delta \le d(o, U)-n  < \Delta \}.
\end{equation}
We say that the horoball growth of $G$ is of \textit{purely exponential type} if $$\sharp H(o, n, \Delta)\asymp \exp(\g G n)$$ for some $\Delta>0$. It is clear that the   equivalent type of the horoballs growth function  $\sharp H(o, n, \Delta)$ does not depend on the choice of $o$ and $\mathbb U$.

So our main theorem reads.
\begin{thm}\label{mainthm}
Suppose $G$ admits a cusp-uniform action on a proper hyperbolic space $(X,d)$ such that $0<\g G <\infty$ and $G$ is of divergent type. Then the following statements are equivalent:
\begin{enumerate}
\item
$G$ satisfies the DOP condition.  
\item
$G$ has the purely exponential orbit growth.
\item
$G$ has the purely exponential orbit growth in (partial) cones.
\item
$G$ has the purely exponential horoball growth.  
\end{enumerate}
\end{thm}
\begin{rem} (on the proof)
It is possible that $\g G=\infty$, see Example 1 in \cite[Section 3.4]{GabPau}. The direction ``$(4)\Rightarrow(2)$" follows by definition of cusp-uniform actions (Lemma \ref{horoball=orbit});  ``$(2)\Rightarrow(4)$" is proved in Section \ref{Section4} by using the PCP property. The main direction is to prove $``(1)\Rightarrow(2), (3)"$ in Section \ref{Section5}, along the way the converse direction could be a bit easier established.   The refinement in partial cones of $``(3)"$ is due to the partial shadow lemma. 
\end{rem}

The following corollary follows from Theorem \ref{DOPthm}, recovering a result of Roblin in the setting of CAT(-1) spaces. 
\begin{cor}\cite[Th\'eor\`eme 4.1]{Roblin}\label{Roblin}
Suppose that $G$ admits a cusp-uniform action on a simply connected Riemannian manifold with pinched negative curvature such that $G$ is of divergent type. Then $G$ has finite BMS measure if and only if the orbit growth   of $G$ is purely exponential. 
\end{cor}

The following corollary, in particular, the item (3)  of Theorem \ref{mainthm} will be used in \cite[Section 7]{YANG7}. 

\begin{cor}\label{ThmG} Suppose that $G$ satisfies the DOP condition or the stronger PGP property. There exist $r, \epsilon, R, \Delta >0$ such that the following holds  for any $n \ge 0$.

\begin{enumerate}
\item
$\sharp A(o, n, \Delta) \asymp \exp(n\g G)$.
\item
$\sharp (\Omega_{r}(go, n, \Delta) \asymp \exp(n\g
G)$ for any $g \in G$.
\item
$\sharp (\Omega_{r, \epsilon, R}(go, n, \Delta) \asymp \exp(n\g
G)$ for any $g \in G$.

\end{enumerate}
\end{cor}
\begin{rem}
In hyperbolic groups, the statements (1) and (2) were previously known in
\cite[Th\'eor\`eme 7.2]{Coor} and in \cite[Lemma 4]{AL} respectively.
Without any assumption on the group $G$, we obtain
analogous results for word metrics on a relatively hyperbolic group $G$ in \cite{YANG7}.  
\end{rem} 
 
It is worth pointing out that, in \cite[Th\'eor\`eme 4.1]{Roblin},  a precise asymptotic formula of $\sharp A(go, n, \Delta)$ was obtained instead of a bi-Lipschitz inequality in Corollary \ref{Roblin}. However, Corollary \ref{ThmG} is sharp in our context:  there exists a generating set $S$ of $G=PSL(2, \mathbb Z)$ acting on $\Gx$ such that the limit of $\displaystyle\frac{\sharp A(go, n, \Delta)}{\exp(\g G n)}$ does not exist. See \cite[Section 3]{GriH} for related discussions.  On the other hand, in many real applications, a coarse formula of $\sharp A(go, n, \Delta)$ works equally well so the generality of our result allows potential applications in a broader setting, for instance, of coarse geometry.
 
\subsection{Organization of paper}
Section \ref{Section2} discusses some dynamical properties  of a cusp-uniform action on boundaries and a notion of transition points relative to a contracting system. Theorem \ref{ThmP} is proved in Section \ref{Section3}, which is used to show the partial shadow lemma. Sections \ref{Section4} and \ref{Section5} prepare necessary ingredients in the proof of Theorem \ref{mainthm} in Section \ref{Section6}.   

The large portion of this paper is essentially an extraction from a previous version of the paper \cite{YANG7}, where Corollary \ref{ThmG} was proved under the assumption that $G$ has the parabolic gap property. Some results in Section \ref{Section3} found the analogous ones in \cite{YANG7}. We, however, include the detailed proof to address the differences and also make this paper independent.
 
\ack  
The author is grateful to Marc Peign\'e for a helpful conversation in 2013 indicating him the results of T. Roblin, which leads eventually to Theorem \ref{mainthm}. Thanks go to the referee for many useful remarks and suggestions on improving the exposition. 

\section{Preliminaries}\label{Section2}

\subsection{Notations and Conventions}

Let $(X, d)$ be a geodesic metric space. We collect some notations
and conventions used globally in the paper.
\begin{enumerate}
\item
$B(x, r) := \{y: d(x, y) \le r\}$ and $N_r(A): = \{y \in X: d(y, A) \le r \}$ for a subset $A$ in $X$.
\item
$\diam {A}$ denotes the diameter of $A$ with respect to $d$.
\item
All the paths we consider
are rectifiable. Let $p$ be a rectifiable path  in $X$ with arc-length parametrization.  Denote by $\len (p)$ the length
of $p$.  Then $p$ goes from  the initial endpoint $p_-$ to the terminal endpoint $p_+$.  Let $x, y \in p$ be two points which are given by parametrization. Then denote by $[x,y]_p$ the parametrized
subpath of $p$ going from $x$ to $y$.
\item
Given a property (P), a point $z$ on $p$ is
called the \textit{first point} satisfying (P) if $z$ is among the
points $w$ on $p$ with the property (P) such that $\len([p_-, w]_p)$
is minimal. The \textit{last point} satisfying (P) is defined in a
similar way.

\item
For $x, y \in X$, denote by $[x, y]$ a geodesic $p$ in $X$ with
$p_-=x, p_+=y$. Note that the geodesic between two points is usually
not unique. But the ambiguity of $[x, y]$ is usually made clear or
does not matter in the context.

\item
Let $p$ a path and $Y$ be a closed subset in $X$ such that $p \cap Y
\neq \emptyset$. So the \textit{entry} and \textit{exit} points of
$p$ in $Y$ are defined to be the first and last points $z$ in $p$
respectively such that $z$ lies in $Y$.


\item
Let $f, g$ be two real-valued functions with domain understood in
the context. Then $f \prec_{c_1, c_2, \cdots, c_n} g$ means that
there is a constant $C >0$ depending on parameters $c_i$ such that
$f < Cg$.  And $f \succ_{c_1, c_2, \cdots, c_n} g,  f \asymp_{c_1,
c_2, \cdots, c_n} g$ are used in a similar way.  
\end{enumerate}

\subsection{Hyperbolic spaces}\label{HypSSection}
We briefly recall some basics of hyperbolic spaces and direct the reader to \cite{Gro}\cite{GH} for more details.
Recall that in a geodesic triangle, two points $x,  y$ in sides $p$ and $q$ respectively are
called \textit{congruent} if $d(x, o)=d(y, o)$ where $o$ is the
common endpoint of $p$ and $q$. Define the Gromov product $(x, y)_z=(d(x, z)+d(y, z)-d(x, y))/2$ for $x, y, z\in X$. We make use of the following definition of hyperbolic spaces.
\begin{defn}
A geodesic space $(X, d)$ is called \textit{$\delta$-hyperbolic} for $\delta \ge 0$ if any geodesic triangle is \textit{$\delta$-thin}: let $p, q$ be any two sides such that $o:=p_-=q_-$. Then a point $x$ in $p$ such that $d(x, p_-)\le (p_+,q_+)_o$ is $\delta$-close to a congruent point in $q$. 
\end{defn}
 \begin{rem}
Roughly, the Gromov product $(p_+,q_+)_o$ is the \textit{fellow-travel} time of $p$ and $q$ staying within $\delta$-neighbourhoods before diverging rapidly in opposite directions. It measures also the distance from $o$ to the side $[p_+, q_+]$, up to a bounded error.
 \end{rem}

A useful corollary follows from $\delta$-hyperbolicity. 
\begin{lem}\label{thintri}
Let $p, q$ be two geodesics such that $p_-=q_-$ then any point $x\in p$ with $d(x, p_+)>d(p_+, p_+)$ has at most $\delta$-distance to $q$.  
\end{lem}



\paragraph{\textbf{Gromov Boundary and Visual metric}} A proper $\delta$-hyperbolic space can be naturally compactified by the \textit{Gromov boundary} $\pX$ which is the set of  {asymptotic} geodesic rays: two rays are \textit{asymptotic} if they have finite Hausdorff distance. By $\delta$-thin triangle property, we deduce that for any $\xi\ne \zeta \in \pX$, there exists a bi-infinite geodesic $\gamma$ such that two half rays are asymptotic to $\xi, \zeta$ respectively.  The choice of $\gamma$ might not be unique, but different ones have a uniform Hausdorff distance depending on $\delta$ only.   We thus define $d(o, [\xi, \zeta])$ to be the distance from $o$ to any geodesic $\gamma$, up to a bounded error. In a geometric sense, the value $d(o, [\xi, \zeta])$ measures the fellow-travel time of $[o, \xi]$ and $[o, \zeta]$ before departing in opposite directions.

We fix a basepoint $o$ and a small parameter $a>0$ close to $0$. A \textit{visual metric} $\rho_{o, a}$ could be constructed on $\pX$ such that the following holds
$$
\rho_{o, a}(\xi, \zeta) \asymp \exp(-a \cdot d(o, [\xi, \zeta]))
$$  
for any $\xi\ne \zeta \in \pX$.

In a $\delta$-hyperbolic space, we could define the \textit{limit set} $\Lambda Y$ of an unbounded subset $Y$ to be the intersection of $\pX$ with the topological closure of $Y$ in $X\cup \pX$. The topological closure of $Y$ in $X$ shall be denoted by $\partial Y$.  The \textit{limit set} of a subgroup $H$ is defined as that of any $H$-orbit in $X$.

\subsection{Dynamical formulation of cusp-uniform actions}
Sometimes, it is convenient to take a dynamical point of view to study the cusp-uniform action. We take \cite{Tukia}, \cite{Bow2} as general references for these dynamical notions.

\begin{defn}  
Let $T$ be a compact metrizable space on which a group $G$ acts by
homeomorphisms.  
\begin{enumerate}
\item
The action of $G$ on $T$ is a \textit{convergence group action} if the induced
group action of $G$ on the space of distinct triples over $T$ is
proper. The \textit{limit set} $\Lambda(\Gamma)$ of a subgroup $\Gamma
\subset G$ is the set of accumulation points of all $\Gamma$-orbits
in $T$. 
 
\item
A point $\xi \in T$ is called \textit{conical} if there are a
sequence of elements $g_n \in G$ and a pair of distinct points $a, b
\in T$ such that the following holds
$$g_n(\xi, \zeta) \to (a, b),$$
for any $\zeta \in T \setminus \xi$. The set of all conical points is denoted by $\Lambda^cG$.
\item
A point $\xi \in T$ is called \textit{bounded parabolic} if the
stabilizer $G_\xi$ of $\xi$ is infinite, and acts properly
and cocompactly on $T\setminus \xi$. 
\item
A convergence group action of $G$ on $T$ is called \textit{geometrically
finite} if every point $\xi \in T$ is either a conical point or a
bounded parabolic point.
\end{enumerate}
\end{defn}
Assume that $G$ has a cusp-uniform action on a proper hyperbolic space $(X, d)$. It is well-known that $G$ acts as a convergence group action on $\pX$.  The limit set of a subgroup $H$ in a convergence group sense coincides with the limit set of $Ho$ in $X$ for any $o\in X$. Moreover, the cusp-uniform action is the same as the geometrically finite convergence action in the following sense. 
\begin{thm}\cite{Bow1}\cite{Yaman}
If $G$ admits a cusp-uniform action on a proper hyperbolic space $(X, d)$, then $G$ acts geometrically finitely on $\Lambda G \subset \pX$. Conversely, if $G$ acts geometrically finitely on $T$ then $G$ admits a cusp-uniform action on a proper hyperbolic space $(X, d)$ such that $T=\pX$.
\end{thm}
\begin{rem}
The first statement was proved by Bowditch in \cite{Bow1}; the second one was shown by Yaman in \cite{Yaman}.
\end{rem}

It is a useful fact that the stabilizer $G_U$ of $U \in \mathbb U$ acts cocompactly on the (topological) boundary $\partial U$ in $X$. Note that the boundary $\Lambda U$ at infinity of a horoball $U$ in $X \cup \pX$
consists of a bounded
parabolic point fixed by $G_U$. 

Let $\xi$ be a conical point in $\pX$. Then for any $\epsilon >0$,
any geodesic ray ending at $\xi$ exits the $\epsilon$-neighborhood of
any horoball $U \in \mathbb U$ which the geodesic enters into.

The following lemma is clear by the definition of cusp-uniform actions and the above observation.
\begin{lem}[Conical points]\label{charconical}
There exists a constant $r>0$ with the following property.

A point $\xi \in \Lambda G$ is a conical point if and only if there exists a sequence of elements $g_n \in G$ such that for any
geodesic ray $\gamma$ in $X$ with $\gamma_+=\xi$ and $\gamma_-=o\in X$, we have
$$\gamma \cap B(g_no, r) \neq
\emptyset$$ for all but finitely many $g_n$.
\end{lem}

\subsection{Transition points}
A subset $Y$ is called \textit{$\epsilon$-quasiconvex} for  $\epsilon>0$ if any geodesic with two endpoints in $Y$ lies in $N_{\epsilon}(Y)$.
In this subsection, we discuss a notion of transitional points which are defined with respect to a system $\mathbb Y$  of uniformly quasiconvex subsets with \textit{bounded intersection}:   for any $\epsilon >0$
there exists $R =R(\epsilon)>0$ such that
$$
\diam{N_\epsilon(Y) \cap N_\epsilon(Y')} < R$$ for any two distinct
$Y\ne  Y' \in \mathbb Y$. 

It is well-known that in a cusp-uniform action, the $G$-finite system $\mathbb U$ of horoballs are uniformly quasiconvex with bounded intersection (cf. \cite[Sections 5 \& 6]{Bow1}). Moreover, we make the following convention throughout this paper. 

\begin{conv}\label{ConvContracting}
Let $\mathbb U\subset \mathbb Y$ be a $G$-finite system of quasiconvex subsets with bounded intersection such that, for each $Y\in \mathbb Y\setminus \mathbb U$, the stabilizer $G_Y$ acts co-compactly on $Y$. We consider below the transition points defined with respect to $\mathbb Y$.
\end{conv}

Before moving on, let's describe a typical example of $\mathbb Y$ in Convention \ref{ConvContracting}. We could consider an ``extended" relative hyperbolic structure of $G$,
which is obtained by adjoining into $\mathcal P$ a collection of
subgroups $\mathcal E$. This can be done in the following way. Let $h
\in G$ be a hyperbolic element. Denote by $E(h)$ the stabilizer in
$G$ of the fixed points of $h$ in $\pX$. Then $\mathcal E = \{g
E(h)g^{-1}: g \in G\}$ gives such an example. See \cite{Osin2} for more
detail.

Let $\mathcal C(E)$ denote the convex hull in $X$ of $\Lambda(E)$ for each
$E \in \mathcal E$. Denote $\mathbb E = \{\mathcal C(E): E \in \mathcal E\}$.
Then $\mathbb Y:=  \mathbb U \cup \mathbb E$ satisfies Convention \ref{ConvContracting}. 

In the following definition, we state an abstract formulation of a
notion of transition points, which was introduced originally  by
Hruska \cite{Hru} in Cayley graphs.
\begin{defn}
 Fix $\epsilon, R>0$.  Given a path $\gamma$, we say that a point $v$ in $\gamma$ is called \textit{$(\epsilon,
R)$-deep} in $Y \in \mathbb Y$ if it holds that $\gamma \cap B(v, R) \subset
N_\epsilon(Y).$ If $v$ is not $(\epsilon, R)$-deep in any $Y \in
\mathbb Y$, then $v$ is called an \textit{$(\epsilon, R)$-transition
point} in $\gamma$.
\end{defn}

\subsection{Projections and contracting property}
Given a subset $Y$ in a metric space $X$,  the \textit{projection} $\proj_Y(x)$ of a point
$x$ to $Y$ is the set of nearest points in the closure of $Y$ to
$x$. Then for $A \subset X$ define $\proj_Y(A) = \cup_{a\in A
}\proj_Y(a)$.

The following result is standard, with the proof left to the interested reader.

\begin{lem}\label{contracting}
For any $C>0$, there exists a constant $M=M(C)>0$  such that  the following hold for a $C$-quasiconvex subset $Y$.
\begin{enumerate}
\item
If $\gamma$ is a geodesic in $X$ with $N_C(Y) \cap \gamma =
\emptyset$, then $\diam{\proj_Y(\gamma)} < M$.
\item
$\diam{\proj_Y(x)}\le M$ for any $x\in X$.
\item
$d(\proj_Y(\gamma_-), \gamma) \le M$ for  a geodesic 
$\gamma$ with $\gamma_+ \in Y$.  
\end{enumerate}
\end{lem}
 
The first item is called \textit{contracting property}, which in fact characterizes the quasiconvexity of a subset in hyperbolic spaces.  As a matter of fact, most of results hold for contracting subsets in a general (not necessarily hyperbolic) metric space.  

The main result of this subsection is the following analogue of Lemma 2.18 in \cite{YANG7}.  Here we give a proof making use of $\delta$-hyperbolicity. It can be also proved by making use of contracting property of $Y$ as in \cite{YANG7}.
\begin{lem}\label{PairTransX}
There exist $\epsilon, R>0$ with the following property. For any
$r>0$, there exist  $D=D(\epsilon, R), L =L(\epsilon, R, r)>0$ with
the following property.

Let $\alpha, \gamma$ be two geodesics in a $\delta$-hyperbolic space $X$ such that
$\alpha_-=\gamma_-, \; d(\alpha_+, \gamma_+) < r$. Take an
$(\epsilon, R)$-transition point $v$ in $\alpha$ such that $d(v,
\alpha_+) >L$. Then there exists an $(\epsilon, R)$-transition point
$w$ in $\gamma$ such that $d(v, w) \le D$.
\end{lem}

\begin{proof} 
Since $d(v,
\alpha_+) > L > r\ge d(\gamma_+, \alpha_+)$,  there exists $w\in \gamma$ such that 
\begin{equation}\label{dvwEQ}
d(v, w) \le \delta
\end{equation}
by Lemma \ref{thintri}.  

Let $\epsilon_0>0$ be the quasiconvexity constant for all $Y\in \mathbb Y$, and $\epsilon=2\delta+ \epsilon_0$  and $R>\mathcal R(\epsilon)$, where $\mathcal R$ is the bounded intersection function of $\mathbb Y$.
Assume that $w$ is $(\epsilon, R)$-deep in some $Y \in \mathbb
Y$. Otherwise, $w$ is an $(\epsilon, R)$-transition point and the proof is
done.

\begin{claim}
Let $x, y$ be the entry and exit points of $\gamma$ in $N_\epsilon(Y)$ respectively. Then $x, y$ are $(\epsilon, R)$-transition points.
\end{claim}
\begin{proof}[Proof the claim]
 Indeed,  by quasiconvexity of $Y$, $[x, y]$ lies in $N_\epsilon(Y)$. Since $R>\mathcal R(\epsilon)$, it follows by $\mathcal R$-bounded intersection that no subsegment of length $R$ in $[x, y]$ can be contained in $N_\epsilon(Y')$ for $Y'\ne Y\in \mathbb Y$. Thus $x, y$ are $(\epsilon, R)$-transition points. 
\end{proof}
To complete the proof, we shall prove  the following 
\begin{equation}\label{dxyEQ}
d(v, x), d(v, y) > R+\epsilon+4\delta,
\end{equation}
is impossible. Suppose now (\ref{dxyEQ}) holds. 

If $d(y, \gamma_+)> r$, let $z=y$; otherwise, let $z\in \gamma$ such that $d(z, \gamma_+) = r$. In both cases, there exists $\tilde z\in \alpha$ such that $d(z, \tilde z)\le \delta$ by Lemma \ref{thintri}. Let $\tilde x \in \alpha$ such that $d(x, \tilde x)\le \delta$ by Lemma \ref{thintri}. Hence, we obtain that $\tilde x, \tilde z\in N_{\delta+\epsilon}(Y)$. 

We now give a lower bound on $d(v, \tilde z)>R+\epsilon+3\delta$. In the former case,   $d(v, \tilde z)>d(v,  z)-d(z, \tilde z)>R+\epsilon+3\delta$ by  (\ref{dxyEQ}). For the later case,  since $d(v, \alpha_+)>L$ by hypothesis, we have $d(v, \gamma_+)>d(v, \alpha_+)-d(\gamma_+, \alpha_+)>L-r$.  Setting $L:=2r+R+\epsilon+4\delta$, we obtain 
$$
\begin{array}{ll}
d(v, \tilde z)&> d(v, z)-\delta>d(v, \gamma_+)-d(\gamma_+, z)-\delta\\
&>L-2r-\delta>R+\epsilon+3\delta.
\end{array}
$$ 

Let $x', y'\in Y$ such that $d(\tilde x, x'), d(\tilde z, z')\le \delta+\epsilon$. By thin-property of quadrangle,   every side lies in the $2\delta$-neighborhood of the other three sides. Note that $d(\tilde x, v), d(\tilde z, v)>R+3\delta+\epsilon$. So each point in the neighborhood $B(v, R)\cap [\tilde x, \tilde z]$ of $v$ has at least a $2\delta$-distance to $[\tilde x, x']$ and $[\tilde z, z']$, and thus   lies in $N_{2\delta}([x', y'])$. Hence, $B(v, R)\cap [\tilde x, \tilde z]\subset N_{\epsilon}(Y)$ implying that $v$ is $(\epsilon, R)$-deep in $Y$. This   contradicts to the assumption that $v$ is $(\epsilon, R)$-transitional, so \ref{dxyEQ}) is not true.

Since $x, y$ are $(\epsilon, R)$-transition points in $\gamma$, we have $d(w, \{x, y\})\le d(w, v)+d(v, \{x, y\}) \le D:=R+\epsilon+5\delta$. The proof is complete.
\end{proof}
 
\subsection{Partial sums in Poincar\'e series} In some computations, we often need deal with partial sums in Poincare series.  
So we introduce the following two closely related series for simplying notations: 
\begin{equation}\label{SeriesR}
\mathcal{S}_{Y}(z, w,R)=\sum\limits_{h \in G_Y}^{d(z, hw)\ge R} \exp(-\g Gd(z, hw)),
\end{equation}
and
\begin{equation}\label{SeriesDelta}
\mathcal{A}_{Y}(z, R, \Delta) = \sum\limits_{n\ge R}\sharp A_Y(z, n, \Delta) \cdot \exp(-\g G n),
\end{equation}
for $Y\in \mathbb Y$, where 
$$
A_Y(o, n, \Delta) := A(o, n, \Delta) \cap G_Y.
$$
The following elementary relations will be useful. 
 
\begin{lem}\label{ConvertSA}
\begin{enumerate}
\item
$\mathcal{S}_{Y}(o, o,R-\Delta) \asymp_\Delta \mathcal{A}_{Y}(o, R, \Delta)$
\item
There exists $K=K(r)$ such that 
$$\mathcal{S}_{Y}(o, o, R+K) \le \mathcal{S}_{Y}(z, w, R) \le  \mathcal{S}_{Y}(o, o, R-K),$$
for any $z, w \in N_r(Y)$ and   $o\in \partial Y$.  
\end{enumerate}
\end{lem}
\begin{proof}
\textbf{(1).} Observe that
$$\sum\limits_{h \in G_Y}^{d(o, ho)>R-\Delta}  \exp(-\g Gd(o, ho)) \asymp_\Delta \sum\limits_{i\ge R}  A_Y(o, i, \Delta)\exp(-\g G\cdot i)$$
for any $n>0$. 
The ``$\prec_\Delta$'' direction is trivial. The other direction 
$$\frac{1}{N} \sum\limits_{i\ge R}  A_Y(o, i, \Delta) \exp(-\g G i)\le \sum\limits_{h \in G_Y}^{d(o, ho)\ge R-\Delta}  \exp(-\g Gd(o, ho)),$$
follows from the fact that at most $N:=\sharp N(o, 2\Delta)$ sets $A_Y(o, i, \Delta)$ contain  $ho$. 

\textbf{(2).} 
By Convention \ref{ConvContracting}, $G_Y$ acts co-compactly either on  $\partial Y$ or on $Y$ with a fundamental domain of diameter at most $M$. So there exist $h_1, h_2\in G_Y$ such that $d(h_1o, z), d(h_2o, w)\le M+r$. If $d(z, hw)\ge R$, then $d(o, h_1^{-1}hh_2\cdot o)\ge R-2(M+r)$. Setting $K=2(M+r)$, we obtain that 
$$
\mathcal{S}_{Y}(z, w,R) \le \mathcal{S}_{Y}(o, o,R-K),
$$
for any $R>0$. The proof is complete.
\end{proof}

The next lemma states that the series $\mathcal{A}_{U}$ is almost conjugacy invariant. This   will be useful in Section \ref{Section6}.

\begin{lem}\label{biginclusion}
For any $r>0$, there exists $K=K(r)>0$ with the following property.  
Let $x \in N_r(U), y\in N_r(V)$   where $U=gV \in \mathbb Y$ for $g\in G$.  Then for any $\Delta, R>0$,
\begin{equation}\label{UVxyEQ}
\mathcal{A}_{V}(y, R+K, \Delta) \prec_\Delta \mathcal{A}_{U}(x, R, \Delta)  \prec_\Delta \mathcal{A}_{V}(y, R-K, \Delta).
\end{equation}
\end{lem}
\begin{proof}
Since   $A_U(x, i, \Delta)=A_V(gx, i, \Delta)$,  we have 
$$
\begin{array}{ll}
 \mathcal{A}_{U}(x, R, \Delta)  \asymp_\Delta  \mathcal{S}_{V}(gx, gx, R-\Delta)  \\
\end{array}
$$
by Lemma \ref{ConvertSA}(1). Apply Lemma \ref{ConvertSA}(2) gives a constant $K=K(r)$ such that
$$
\mathcal{S}_{V}(y, y, R-\Delta+K) \le  \mathcal{S}_{V}(gx, gx, R-\Delta) \le   \mathcal{S}_{V}(y, y, R-\Delta-K). 
$$
So by Lemma \ref{ConvertSA}(1): 
$$
\mathcal{A}_{V}(y, R+K, \Delta) \prec_\Delta \mathcal{S}_{V}(gx, gx, R-\Delta)  \prec_\Delta \mathcal{A}_{V}(y, R-K, \Delta)
$$
completing the proof of the lemma. 
\end{proof}

\section{Quasi-conformal densities}\label{Section3}

In this section,  assume  that $G$ has a cusp-uniform action on $X$ such that $G$ is of divergent type (results of Coornaert in Subsection \ref{Coornaert} even hold without this assumption).    We fix the basepoint $o \in X$, and a system $\mathbb Y$ in the convention \ref{ConvContracting} with respect to which transitional points are defined.

\subsection{Partial shadows and cones}

These notions were introduced in \cite{YANG7}.

\begin{defn}[Shadow and Partial Shadow]
Let $r, \epsilon, R \ge 0$ and $g \in G$. The \textit{shadow $\Pi_r(go)$} at $go$ is the set of points
$\xi \in \pG$ such that there exists SOME geodesic $[o, \xi]$
intersecting $B(go, r)$.

The \textit{partial shadow $\Pi_{r, \epsilon,
R}(go)$} is the set of points $\xi \in \Pi_r(go)$ where, in addition, the geodesic $[o, \xi]$ as above contains an $(\epsilon, R)$-transition point $v$ in $B(go, 2R)$.
\end{defn}

Inside the space $X$, the (partial) shadowed region motivates the notion of a (partial) cone. 
\begin{defn}[Cone and Partial Cone]\label{defnpcone}
Let $g \in G$ and $r \ge 0$. The \textit{cone $\Omega_r(go)$} at $go$
is the set of elements $h$ in $G$ such that there exists SOME
geodesic $\gamma=[o, ho]$ in $X$ such that $\gamma \cap B(go, r) \neq
\emptyset$.

The \textit{partial cone
$\Omega_{r, \epsilon, R}(go)$} at $go$ is the set of elements $h \in
\Omega_r(go)$ such that one of the following statements holds.
\begin{enumerate}
\item $d(o, ho) \le d(o, go) + 2R,$
\item
The geodesic $\gamma$ as above contains an $(\epsilon,
R)$-transition point $v$ such that $d(v, go) \le 2R$.
\end{enumerate}
 
\end{defn}

\subsection{Patterson-Sullivan measures}\label{Coornaert}
The aim of this subsection is to recall some basic results of Patterson-Sullivan measures established by Coornaert in \cite{Coor} for general discrete group actions on hyperbolic spaces.

A Borel measure $\mu$ on a topological space $T$ is \textit{regular}
if $\mu(A) = \inf\{\mu(U): A \subset U, U$ is open$\}$ for any Borel
set $A$ in $T$. The $\mu$ is called \textit{tight} if $\mu(A) =
\sup\{\mu(K): K \subset A, K$ is compact$\}$ for any Borel set $A$
in $T$.

Recall that \textit{Radon} measures on a topological space $T$ are
finite, regular, tight and Borel measures. It is well-known that all
finite Borel measures on compact metric spaces are Radon. Denote by
$\mathcal M(\Lambda G)$ the set of finite positive Radon measures on
$\Lambda G$. Then $G$ possesses an action on $\mathcal M(\Lambda G)$ given by
$g_*\mu(A) = \mu(g^{-1}A)$ for any Borel set $A$ in $\Lambda G$.

Endow $\mathcal M(\Lambda G)$ with the weak-convergence topology. Write
$\mu(f) = \int f d\mu$ for a continuous function $f \in C^1(\Lambda G)$.
Then $\mu_n \to \mu$ for $\mu_n \in \mathcal M(\Lambda G)$ if and only if $\mu_n(f)
\to \mu(f)$ for any $f \in C^1(\Lambda G)$,  equivalently, if and only if,
$\liminf\limits_{n \to \infty}\mu_n(U) \ge \mu(U)$ for any open set
$U \subset \Lambda G$.

\begin{defn}
Let $\sigma \in [0, \infty[$.    A   map
$$\mu: G \to \mathcal M(\Lambda G), \;g \to \mu_g$$ is called \textit{$G$-equivariant}
if $\mu_{hg}(A) = h_*\mu_{g}(A)$ for any Borel set $A \subset
\Lambda G$.
 It is a \textit{$\sigma$-dimensional quasi-conformal density} if for any $g,
h \in G$ the following holds
\begin{equation}\label{cdensity}
\frac{d\mu_{g}}{d\mu_{h}}(\xi) \asymp \exp(-\sigma B_\xi (go, ho)),
\end{equation}
 for $\mu_h$-a.e. points $\xi \in \Lambda G$.

\end{defn}

\begin{rem}
Denote $\nu = \mu_1$, where $1 \in G$ is the group identity.  Define $g^*\nu(A) = \nu(gA)$.
By the equivariant property of $\mu$, we obtain the following
\begin{equation}\label{qconf}
\frac{dg^*\nu}{d\nu}(\xi) \asymp \exp(-\sigma B_\xi (g^{-1}o, o)),
\end{equation}
for $\nu$-a.e. points $\xi \in \Lambda G$. The inequality (\ref{qconf}) instead of (\ref{cdensity}) is used by some
authors to define the quasi-conformal density, for example, in \cite{Coor}. As $G$ acts minimally on $\Lambda G$ with $\sharp \Lambda G>3$, $G$ has no global fixed point on $\Lambda G$. Then $\mu_g$
is not an atom measure.

\end{rem}

By the equivariant property of $\mu$, we see the following result.
\begin{lem}\cite[Corollaire 5.2]{Coor}
Let $\{\mu_v\}_{v\in G}$ be a $\sigma$-dimensional quasi-conformal density on $\Lambda G$.
Then the support of any $\mu_v$ is $\Lambda G$.
\end{lem}

Here is a well-known construction, due to Patterson \cite{Patt}, of a quasi-conformal density. 
We start by constructing a family of measures $\{\mu_v\}^s_{v\in G}$ supported on $Go$ for any $s >\g G$. If $\PS{G}$ is divergent at $s=\g G$, set
$$\mu^s_v = \frac{1}{\PS{G}} \sum\limits_{g \in G} \exp(-sd(vo, go)) \cdot \dirac{go},$$
where $s >\g G$ and $v \in G$. Note that $\mu^s_1$ is a probability
measure. Choose $s_i \to \g G$ such that $\mu_v^{s_i}$ are convergent in
$\mathcal M(\Lambda G)$. Let $\mu_v = \lim \mu_v^{s_i}$ be the limit
measures, which are so called \textit{Patterson-Sullivan measures} (PS-measures as shorthand).  

In the case that $\PS{G}$ is convergent at $s=\g G$, S. Patterson invented a trick to construct $\mu^s_v$ similarly such that the limit $\mu_v$ are supported on $\pG$. Since in our case $G$ is assumed to be divergent type, we omit the discussion of this case and refer the reader to \cite{Patt}.

\begin{prop}\cite[Th\'eor\`eme 5.4]{Coor}
PS-measures $\{\mu_v\}_{v\in G}$  on $\pG$   is a $\g G$-dimensional quasi-conformal density.
\end{prop}

\begin{lem}[Shadow Lemma]\label{OmbreLem}\cite[Proposition 6.1]{Coor}
Let $\{\mu_v\}_{v\in G}$ be  $\sigma$-dimensional quasi-conformal density on $\pG$ for $\sigma>0$. There exists $r_0 >0$ such that the following holds,
$$
\exp(-\sigma d(go, o))  \prec \mu_1(\Pi_r(go)) \prec_r \exp(-\sigma d(go,
o)),
$$
 for any $g \in G$ and $r >r_0$.
\end{lem}


\begin{lem}[Exponential growth]\label{ballgrowthII}\cite[Proposition 6.4]{Coor}
There exists $\Delta_0 >0$ such that the following holds
$$
\sharp A(o, n, \Delta) \prec_\Delta \exp(n\g G),
$$
for any $n \ge 0$ and $\Delta >\Delta_0$.
\end{lem}

\subsection{PS-mesures have no atoms}
Let $\{\mu_v\}_{v\in G}$ be PS-measures on $\pG$. We shall show that $\mu_v$ has no atoms at parabolic points,
following an argument of Dal'bo-Otal-Peign\'e in \cite[Propositions
1 \& 2]{DOP}.  The first lemma is proved by the same proof of Proposition 1 in \cite{DOP}. 

\begin{lem}\label{convpara}
Let $H$ be a subgroup in $G$ such that $\Lambda H$ is properly
contained in $\Lambda G$. Then for any $o\in X$, $\PS{H}$ is convergent at $s=\g G$. In
particular, if $H$ is divergent, then $\g H < \g G$.
\end{lem}

Recall that  $\pX$ denotes the Gromov boundary of $X$. Given a subset $Y \subset X$, denote by $\Lambda Y$ the boundary of $Y$ at infinity in $\pX$. The limit set of $G_Y$ coincides with $\Lambda Y$, since $G_Y$ acts co-compactly on $Y$ or $\partial Y$ in Convention \ref{ConvContracting}.

The content of next two lemmas are not surprising, but we could not find   explicit statements   so we collect them here at the convenience of the reader.   

\begin{lem}\label{cpctdomain}
Let $P:=G_Y$ be a subgroup in $G$ given in Convention \ref{ConvContracting}. Then $P$ acts co-compactly on $(X\cup \pX)\setminus ( Y\cup \Lambda Y)$ endowed with topology from $X\cup \pX$. 
\end{lem}

\begin{proof}
We first consider the case that $Y$ is a horoball in $\mathbb U$ so $\Lambda Y$ consists of one point $q$.    Since $q \in \pX$ is a bounded parabolic point,  $P$ acts on $\pX\setminus \{q\}$ with a compact fundamental domain $Q \subset \pX\setminus \{q\}$. Consider the set $H(Q)$ in $X$ which is the union of all geodesics with one endpoint in $Q$ and the other one ending at $q$. Then $\hat Q:=\Big(H(Q)\cap (X\setminus Y)\Big)\cup Q$ is compact and $P\cdot \hat Q=(X\cup \pX)\setminus (Y\cup q)$ by construction. This implies that $\hat Q$ is a compact fundamental domain.

Now assume that $P$ acts co-compactly on $Y$ with a compact fundamental domain $Q$. Let $\hat Q$ be the set of points $x\in X\setminus Y$ such that a projection point of $x$ to $Y$ belongs to $Q$. By construction, $H\cdot   \hat Q=X\setminus  Y$. Consider the topological closure of $\hat Q$ in $X\cup \pX$, still denoted by $\hat Q$, so we have $H\cdot   \hat Q=(X\cup \pX)\setminus  (Y\cup \Lambda Y)$. Hence, in order to prove that $\hat Q$ is a desired compact fundamental domain, it suffices to prove $\hat Q\cap \Lambda Y=\emptyset.$ 

By way of contradiction, assume that $\xi \in \hat Q\cap \Lambda Y$, so there exist $x_n\in  \hat Q\cap X, y_n\in Y$ such that $x_n\to \xi, y_n\to \xi$. By the convergence of Gromov topology (cf. subSection \ref{HypSSection}), we have $[o, x_n]$ fellow travels $[o, y_n]$ as long as possible as $n\to \infty$, for a basepoint $o\in Q$. Hence, there exist a constant $C$ depending on $\delta$ and a sequence of points $p_n \in [o, x_n]$ such that $d(p_n, [o, y_n])\le C$ and $d(o, p_n)\to \infty$. By the quasiconvexity of $Y$,  we have that $p_n$ stays within a uniform neighborhood of $Y$, say $y_n \in N_C(Y)$ for easy notations. As $d(o, p_n)\to \infty$ and $d(p_n, Y) \le C$, we got a  contradiction since  $x_n$ has   a shortest projection point in a compact set $Q \subset Y$.  Thus, $\hat Q\cap \Lambda Y=\emptyset$, completing the proof of lemma.
\end{proof}

To prove no atoms at parabolic points, we need the following observation. 
\begin{lem}\label{projection}
Let $U, V \subset X$ such that $\Lambda U \cap \Lambda V=\emptyset$. Then there exists a constant $M=M(U, V)>0$ such that
$\diam{\proj_V(U)}\le M$.
\end{lem}
\begin{proof}
Fix a point $v \in V$.  We claim that there exists a finite number $M>0$ such that $d(v, [u, \bar u]) \le M$ for any $u \in U$ and a projection point $\bar u$ to $V$. Indeed, assume, by way of contradiction, that there exist infinitely many pairs $(u_n, \bar u_n) \in U \times V$ such that $d(v, [u_n, \bar u_n]) \ge n$ for $n \in \mathbb N$. Without loss of generality, assume that $u_n, \bar u_n$ converge to $\xi, \eta\in \pX$ respectively. Since $d(v, [u_n, \bar u_n]) \ge n$, we see that $\xi=\zeta$ by definition of visual metric  (cf. subSection \ref{HypSSection}.) However, this is a contradiction, as $\Lambda U\cap \Lambda V=\emptyset$. Thus, the claim is proved.

Let $x\in [u, \bar u]$ so that $d(v, x)=d(v, [u, \bar u])$. Since $\bar u$ is a shortest point in $V$ from $u$, we see that $d(x, \bar u) \le d(x, v)$.  By the claim above, $d(v, \bar u)\le d(v, x)+d(x, \bar u)\le 2M$ for any $u\in U$. Hence, $\diam{\proj_V(U)}\le 4M$. 
\end{proof} 

By Lemma
\ref{projection}, the following result is proved with the almost same proof as \cite[Lemma 4.10]{YANG7}.  
\begin{lem}\label{parabnoatom2}
Assume that $G$ is of divergent type. Then $\{\mu_v\}_{v\in G}$ gives zero measure to $\Lambda(P)$, where $P:=G_Y$ is given by Convention \ref{ConvContracting}. In particular, PS measures have no atoms at bounded parabolic points.
\end{lem}
\begin{proof}[Sketch of the proof]We only focus on the difference with \cite[Lemma 4.10]{YANG7} and refer the reader there for more details.
 
By Lemma \ref{cpctdomain}, $P$ acts on $(X\cup \pX)\setminus (Y\cup \Lambda Y)$ with a compact fundamental domain $K$. We can further assume that   the boundary of $K$ in $X\cup \pX$ is $\mu_v$-null for some (hence any) $v\in
G$. For convenience, we may choose the basepoint $o\in \partial Y$ if $Y$ is a horoball,  otherwise $o\in Y$.    Up to a translation by an element in $P$,
we further assume that $o \in \proj_{Y}(K\cap Go)$.   Let $$V_n = \bigcup\limits_{d(o, po) > n}^{p\in P} p K.$$  Then $V_n \cup \Lambda Y$
is a decreasing sequence of open neighborhoods of $\Lambda Y$. Since the
boundary of $V_n$ is $\mu_1$-null, it follows that $\mu_1^{s}(V_n)
\to \mu_1(V_n)$ for each $V_n$, as $s \to \g G$.


With Lemmas \ref{contracting} and \ref{projection} in hand, we can proceed as \cite[Lemma 4.10]{YANG7} to get
$$
\begin{array}{rl}
\mu_1^s(V_n) & \prec   \mu_1^s(K) \cdot\left( \sum\limits_{d(o, po) > n}^{p\in P}  \exp(-sd(o, po))
\right).
\end{array}
$$

Letting $s \to \g G$ we have $\mu_1^s(V_n) \to \mu_1(V_n)$. By Lemma \ref{convpara}, $\PS{P}$ is convergent at $s=\g G$ so that $\mu_1(V_n) \to 0$ as $n\to \infty$. Hence, $\mu_1(\Lambda P) =0$.
\end{proof}

Fix $\Delta>1$ and consider $G = \cup_{i\ge 1} A(o, i, \Delta)$, where $A(o, i, \Delta)$ is given in (\ref{OrbitFunction}). By Lemma
\ref{charconical}, we have that $\Lambda^cG = \Lambda_r$ for any fixed $r \gg 0$, where
\begin{equation}\label{conicalset}
\Lambda_r:=\bigcap_{n=1}^{\infty} \left(\bigcup_{g\in \cup_{i\ge n} A(o, i, \Delta)}
\Pi_r(go)\right). 
\end{equation}  In
other words, a conical point is shadowed infinitely many times
by the orbit $Go$.
 
\begin{lem}\label{conicnoatom}
Conical points are not atoms of PS-measures.
\end{lem}
\begin{proof}
Note that a conical point $\xi$ lies in infinitely many shadows $\Pi_r(go)$ for $g\in G$. As $\mu_1(\Pi_r(go))\to 0$ as $d(o, go)\to \infty$ we have that $\mu_1(\xi)=0$. 
\end{proof}

As $G$ acts  geometrically finitely on $\Lambda G$, there exist only bounded
parabolic points and conical points in $\Lambda G$. Hence, Lemmas
\ref{parabnoatom2} and \ref{conicnoatom} together prove the
following proposition, where the ``moreover" statement is proved in \cite[Appendix: Proposition A.4]{YANG7}.

\begin{prop}\label{PSMeasureX}
Assume that $G$ is of divergent type. Then
PS-measures $\{\mu_v\}_{v\in G}$ are $\g G$-dimensional quasi-conformal density
without atoms. Moreover, $\mu$ is unique and ergodic.
\end{prop}

\subsection{Partial Shadow Lemma}

Recall that in  Convention \ref{ConvContracting},   $\mathbb Y$ is a $G$-finite   system of quasiconvex subsets with   bounded intersection   so that $\mathbb U\subset \mathbb Y$, for which we consider the transition points. We shall prove a variant of Shadow Lemma called \textit{Partial Shadow Lemma} with respect to the so-defined transition points.    

Let's prepare some preliminary results. The following lemma will be used crucially in the sequel.

\begin{lem}\label{convps2}
For any $\varepsilon, r >0$, there exists $R=R(\varepsilon, r)>0$ such
that the following holds
$$
\mathcal{S}_{Y}(z, w,R)< \varepsilon,
$$
for any $Y \in \mathbb Y$ and any $z, w \in N_r(\partial Y)$.
\end{lem}
\begin{proof}
Since $\mathbb Y/G$ is finite by Convention \ref{ConvContracting}, it is enough to prove the lemma for finitely many $Y\in \mathbb Y$. 

Since $\mathbb Y$ has the bounded intersection property, we see that the topological closure $\Lambda Y$ of $Y$ in $\pX$ is a proper subset  so that the limit set of $G_Y$ is also proper. By Lemma \ref{convpara}, we have
$$
\mathcal{S}_{Y}(o, o, 0)< \infty,
$$
which clearly concludes the proof by Lemma \ref{ConvertSA} (2).
\end{proof}

For each $Y \in \mathbb Y$, it is useful to consider the \textit{foot} $o_Y\in \partial Y$ of $o$, which is defined as a projection point of $o$ to $Y$. By Lemma \ref{contracting} (2), the foot is well-defined up to a bounded amount. Since $X\setminus \cup \mathbb U$ is $G$-cocompact by  definition of a cusp-uniform action, the next lemma is straightforward.

\begin{lem}\label{horoball}
There exists a constant $M >0$ with the following property.

For each $Y \in \mathbb Y$, there exists $t_Y \in G$ such that
$d(t_Y o, o_Y) \le M$ and $\partial Y \subset N_M((G_Y t_Y)\cdot
o)$.
\end{lem}


We now state the partial shadow lemma. We remark that the proof replies crucially on the fact that $\mu_1(\Lambda G_Y)=0$ in Lemma \ref{parabnoatom2}, which in turn  is proved using the divergence of the action of $G$ on $X$. 
\begin{lem}[Partial Shadow Lemma]\label{PShadowX}
Let   $\epsilon>0$ be
given by Lemma \ref{PairTransX}. There are constants $r, \epsilon, R\ge 0$ such that the following
holds 
\begin{equation}\label{Fombre2}
\exp(-\g G d(o, go)) \prec \mu_1(\Pi_{r, \epsilon, R}(go)) \prec_r \exp(-\g G d(o, go)),
\end{equation}
for any $g \in G$.
\end{lem}


\begin{proof}
 Given any $g\in G$, there exist $r>M, C_1, C_2>0$ be given by the Shadow Lemma \ref{OmbreLem} such that
\begin{equation}\label{LBND}
C_1 \exp(-\g G d(o, go)) \le \mu_1(\Pi_r(go)) \le C_2 \exp(-\g G d(o, go)).
\end{equation}
 Denote by $\mathbb F$ the set of $Y \in \mathbb Y$ such
that $Y \cap B(go, r + \epsilon) \ne \emptyset$. Since $\mathbb Y$ is locally finite, we have that $\sharp \mathbb F$ is a uniform number depending only on $G$. The choice of the constant $R>0$ will be made in the remainder of proof.


Denote $\Xi:= \Pi_r(go) \setminus \Pi_{r, \epsilon, R}(go)$. 
For any $\xi \in \Xi$, any geodesic $\gamma=[o, \xi]$ does not contain an $(\epsilon,
R)$-transition point in the ball $B(go, 2R)$.

Since $\xi \in  \Pi_r(go)$, we can choose $x\in B(go, r) \cap \gamma$ such that $d(go, x) \le r$.  Assuming that $R>r$, we have that $x$ is $(\epsilon, R)$-deep in some $Y\in \mathbb Y$. Thus, $d(x, Y)\le \epsilon$  and   $d(go, Y) \le d(go, x) +d(x, Y) \le r + \epsilon$, so we have $Y
\in \mathbb F$.

Note that $\mu_1(\Lambda G_Y)=0$ by Lemma \ref{parabnoatom2}. Without loss of generality, assume that $\xi$ lies outside $\Lambda Y$ so that $\gamma$ exits every finite neighborhood of $Y \in \mathbb Y$. Let $z$ be the exit point of $\gamma$ in $N_\epsilon(Y)$. 
 
Furthermore, choose $R>\mathcal R(\epsilon)$, where $\mathcal R$ is the bounded intersection function of $\mathbb Y$.  By the claim in the proof of Lemma \ref{PairTransX}, $z$ is an $(\epsilon, R)$-transition point
in $[o, \xi]$. Since $d(go, x) < r$, we see that $d(z, x) > 2R-r$; otherwise, $B(go, 2R)$ contain the $(\epsilon, R)$-transition point $z$ in $\gamma$, a contradiction.

Let $M>0$ be a constant given by Lemma \ref{horoball}, so there exists $h\in G_Y$ and $t_Y\in G$ such that $d(t_Yo, Y)\le M$ and $d(h\cdot t_Y o, z)\le M\le r$. This implies 
\begin{equation}\label{EQzgo}
\begin{array}{lll}
d(h\cdot t_Y o, go) &>d(z, x)-d(z, h\cdot t_Y o)-d(x, go)\\
&\ge 2R-3r.
\end{array}
\end{equation}

We apply Lemma \ref{convps2} to pair of points $t_Y o, go \in N_{r+\epsilon}(\partial Y)$.
There exists $R_2>0$ depending on $r,\epsilon, \sharp \mathbb F$ such that
\begin{equation}\label{SUM}
\begin{array}{rl}
\mathcal {S}(t_Yo, go, R_2)   \cdot \sharp \mathbb F \cdot \exp(4\g G r) <
C_1/(2C_2).
\end{array}
\end{equation}
Since $z\in [x, \xi]_\gamma$ and $d(go, x)\le r$, we have $d(o, z) +2r \ge d(o, go) + d(go, z)$ by triangle inequality.
Hence, $$
\begin{array}{lll}
d(o, h\cdot t_Y o) \ge d(o, z) -r  \ge d(o, go) + d(go, h\cdot t_Y o) -4r.\\
\end{array}
$$
We assume that $2R-3r\ge R_2$. By (\ref{EQzgo}) and (\ref{SUM}), the following holds:  
\begin{equation}\label{UBND}
\begin{array}{rl}
\mu_1(\Xi) & \le \sum\limits_{Y \in \mathbb F}\left( \sum\limits_{h\in G_Y}^{d(h\cdot t_Yo, go) > R_2} \mu_1(\Pi_r(h\cdot t_Yo))\right) \\
& \le C_2\cdot \sum\limits_{Y \in \mathbb F} \left(\sum\limits_{h\in G_Y}^{d(h\cdot t_Yo, go) > R_2} \exp(-\g G d(o, h\cdot t_Yo))\right) \\
&\le C_2 \cdot \exp(-\g G d(o, go))\cdot \exp(4\g G r) \cdot \sum\limits_{Y \in \mathbb F} \mathcal {S}(t_Yo, go, R_2)  \\
& \le \exp(-\g G d(1, g)) \cdot C_1/2.
\end{array}
\end{equation}
 
Notice that $\mu_1(\Xi) + \mu_1(\Pi_{r, \epsilon, R}(go)) =
\mu_1(\Pi_r(go)) \ge C_1 \exp(-\g G d(o, go))$. So the inequalities
(\ref{LBND}) and (\ref{UBND}) yield $$\mu_1(\Pi_{r, \epsilon,
R}(go)) \ge (C_1/2)\cdot \exp(-\g G d(o, go)).$$ The proof is now complete.
\end{proof}



\section{Horoball growth functions}\label{Section4}
In this section, we prove the directions ``(2)$\Leftrightarrow$(4)" of Theorem \ref{mainthm}.

From now on, we assume that $\pX=\pG$ for simplicity so that  $G$ acts co-compactly on $X\setminus \mathcal U$. 
We fix a constant 
\begin{equation}\label{MDiameter}
M \ge \diam{(X \setminus
\mathcal U)/G},
\end{equation}
which satisfies simultaneously Lemmas \ref{contracting} and \ref{horoball}.

Recall that $H(o, n, \Delta)$ consists of the set of horoballs $U\in \mathbb U$ such that $$\Delta \le d(o, o_U)-n < \Delta,$$ where $o_U$ is a projection point, called  the \textit{foot} of $o$ to $U$. It is obvious that the equivalent type of $\sharp H(o, n, \Delta)$ does not depend on the choice of basepoints. Moreover, it is independent of the choice of horoball system $\mathbb U$ in definition of cusp-uniform actions. Indeed, two different horoball systems are mutually uniformly close by the co-compact action on their complements. By abuse of language, we can speak of   purely exponential  growth of horoball growth, if $\sharp H(o, n, \Delta) \asymp \exp(n \g G)$ for some $\Delta>0$.

\begin{lem}[``(4)$\Rightarrow$(2)'']\label{horoball=orbit}
Suppose that the horoball growth function is purely exponential. Then the orbit growth function of $G$ is purely exponential.  
\end{lem}
\begin{proof}
For each $U\in H(o, n, \Delta)$, by Lemma \ref{horoball}, there exists $t_U \in G$ such that $d(o_U, t_U o)\le M$, and so $-\Delta+M\le d(o, t_Uo)-n < \Delta+M$. Thus, we defined a map from $H(o, n, \Delta) \to A(o, n, M+\Delta)$ by sending $U$ to $t_U$. It suffices to show that this map is uniformly finite-to-one. This follows from the fact that $\mathbb U$ is locally finite: a ball of finite radius intersects only finitely many horoballs. This implies that there are only finitely many $U \in H(o, n, \Delta)$ sending to the same $t_U$.  Hence, $$\sharp A(o, n, \Delta+M) \asymp \exp(n \g G),$$ completing the proof of the lemma.
\end{proof}

The set of $G$-orbits in $\mathbb U$ is finite. Let $$\mathbb{\tilde U}:=\{U_k \in \mathbb U: 1\le k\le m\}$$ be a choice of representatives in each $G$-orbit  among which $d(o, U_i)$ is minimal. 

By the choice of $M\ge \diam{(X\setminus \mathcal U)/G}$, we have that $X\setminus \mathcal U \subset N_M(Go)$. Since $d(o, U_i)$ is minimal, we have $d(o, U_i)\le  M$ for $1\le i\le m$. 

In the remainder of this section, we shall prove the direction ``(2)$\Rightarrow$(4)'' for  the horoball growth function of $G$ of \textit{type} $V\in \mathbb U$:
$$\sharp H_V(o, n, \Delta):= \sharp (H(o, n, \Delta)\cap G\cdot V),$$
for $\Delta>0$.  
We first establish the upper bound on $\sharp H_V(o, n, \Delta)$, which is a consequence of Lemma \ref{ballgrowthII}.
\begin{lem}\label{upperhgrowth}
There exists $\Delta>0$ such that
$\sharp H_V(o, n, \Delta) \prec   \exp(\g G n)$ for any $V\in \mathbb U$.
\end{lem}
\begin{proof}
Let $\Delta>0$ be given by Lemma \ref{ballgrowthII} so $\sharp A(o, n, \Delta+M)\prec   \exp(\g G n).$ For each $U \in H_V(o, n, \Delta)$, there exists  $g\in G$ such that $d(go, U)\le M$ and thus $$-M-\Delta \le d(o, go) -n < \Delta+M,$$
implying  $g\in A(o, n, \Delta+M).$  This sets up a map $$\Phi: H_V(o, n, \Delta) \to A(o, n, \Delta+M)$$ by sending $U$ to $g$.  This map is uniformly finite-to-one, by the fact that $\mathbb U$ is locally finite. Hence,  $\sharp H_V(o, n, \Delta)\prec \sharp A(o, n, \Delta+M)\prec   \exp(\g G n)$. The proves the lemma.
\end{proof}

\begin{prop}[``(2)$\Rightarrow$(4)'']\label{orbit=horoball}
Suppose that the orbit growth function is purely exponential. Then the horoball growth function  of any type $V\in \mathbb U$ is purely exponential.  
\end{prop}
\begin{proof}
By Lemma \ref{upperhgrowth}, it suffices to give a lower bound on $\sharp H_V(o, n, \Delta)$. 
By the purely exponential orbit growth, there exist $C, \Delta>0$ such that 
\begin{equation}\label{EQgrowth}
\sharp A(o, n, \Delta+M)\ge  C  \exp(\g G n) 
\end{equation}
for $n\ge 1$.

We fix a horoball $V\in \mathbb U$, and without loss of generality, assume that $V=U_k\in \mathbb{\tilde U}$ for some $k$.

For each $g \in A(o, n, \Delta)$, there exists  $U\in G \cdot V$  such that $$d(go, U)\le M.$$ So, we consider a map $$\Psi: A(o, n, \Delta) \to G \cdot V\subset \mathbb U$$ by sending $g$ to a choice of $U$ so that $d(go, U)\le M$. By the local finiteness of $\mathbb U$, we see that $\Psi$ is uniformly finite to one: $$\sharp A(o, n, \Delta) \asymp \sharp \Psi(A(o, n, \Delta)).$$ So the idea of the proof is to find a ``large" portion of $A(o, n, \Delta)$ which is sent under $\Psi$ to $H_V(o, n, \Delta')$ for a constant $\Delta'>0$ given below.
To be precise, we will prove the following.
\begin{sublem}
There exists a constant $R>0$ such that the set $\mathbb B$ of elements $g \in A(o, n, \Delta)$ with $d(go, o_U)\ge R$ and $d(go, U)\le M$ satisfies
\begin{equation}\label{BCardEQ}
\sharp \mathbb B \le \displaystyle\frac{C}{2}\exp(n\g G).
\end{equation}
Here $U \in  \mathbb U$ is uniquely determined by $go$. 
\end{sublem}
  
We postpone the proof of the subLemma, and finish the proof of proposition  by assuming it. 

By the subLemma, for each $g \in A(o, n, \Delta+M)\setminus \mathbb B$, we have $d(go, o_U) \le R$ where $U$ is uniquely determined by $go$. Thus, $$-R-M-\Delta<d(o, o_U)-n< \Delta+M+R$$ yielding   $$\Psi(A(o, n, \Delta+M)\setminus \mathbb B) \subset H_V(o, n, \Delta'),$$
where $\Delta':=R+M+\Delta$.    

On the other hand, by (\ref{BCardEQ}) and (\ref{EQgrowth}), we obtain
$$
A(o, n, \Delta+M)\setminus \mathbb B \ge  \displaystyle\frac{C}{2}\exp(n\g G).
$$
Since $\Phi$ is uniformly finite-to-one, we see $$\sharp H_V(o, n, \Delta') \succ \sharp ( A(o, n, \Delta+M)\setminus \mathbb B)\succ \exp(\g G n),$$ proving the lower bound. The proposition is proved.
\end{proof}

Lets now prove the subLemma.

\begin{proof}[Proof of the subLemma]We organize the proof into three steps.

\textbf{Step 1}. We first show that  $U \in \mathbb U$ is uniquely determined by $go$: 
\begin{claim1}
Fix a constant $R_1>0$ as below in (\ref{R0EQ}).   For each $g  \in A(o, n, \Delta)$, there exists at most one horoball  $U \in \mathbb U$  such that $d(go, o_U)\ge R_1$ and $d(go, U)\le M$.
\end{claim1}
\begin{proof}[Proof of the Claim]
Let $x\in \partial U$ such that $d(x, go)\le M$. Assume by contradiction that there exists a distinct $U\ne W\in \mathbb U$ and $y\in \partial W$ such that $d(go, y)\le M$ and $d(go, o_W)\ge R_1$. By Lemma \ref{contracting}, there exists a point $x'\in [o, x]$ such that $d(o_U, x') \le M$.   Since $d(x, go) \le M$,  we obtain by Lemma \ref{thintri} that $$d(x', z_1)\le \delta$$ for some $z_1\in [o, go]$. Hence, 
\begin{equation} \label{ougo1EQ}
d([o, go], o_U)\le d(z_1, x')+d(x', o_U)\le M+\delta.
\end{equation} 
The same argument shows that $d(y', z_2)\le \delta$,  $d(o_W, y') \le M$ for $y'\in [o, y]$ and $z_2\in [o, go]$, and  $$d([o, go], o_W)\le d(z_2, y')+d(y', o_W)\le M+\delta.$$   See Figure \ref{Figure1}. 

For concreteness, assume that $d(o, z_1) \le d(o, z_2)$ and then $z_2\in [z_1, go]$.  Since $go, z_1 \in N_{M+\delta}(U)$, by the quasiconvexity of horoballs, there exists a constant $D=D(M+\delta)>0$ such that $[z_1, go] \subset N_D(U)$. Thus,  $d(z_2, U)\le D$.

Denote $R_0:=D+M+\delta$. We thus proved that $go, z_2\in N_{R_0}(U)\cap N_{R_0}(W)$ so $$
\begin{array}{ll}
\diam{N_{R_0}(U)  \cap N_{R_0}(W)}&\ge d(go, z_2)\ge d(go, o_V)-d(o_V, z_2)\\
&\ge R_1-M-\delta\\
&> \mathcal R(R_0),
\end{array}
$$
where we choose 
\begin{equation}\label{R0EQ}
R_1> \mathcal R(R_0)+(M+\delta),
\end{equation}
where $\mathcal R: \mathbb R_{\ge 0}\to \mathbb R_{\ge 0}$ is the bounded intersection function for the horoball system $\mathbb U$. This is a contradiction to the consequence $\diam{N_{R_0}(U)  \cap N_{R_0}(W)} \le \mathcal R(R_0)$ by $\mathcal R$-bounded intersection of $\mathbb U$. The claim is proved.
\end{proof}

\begin{figure}[htb]
\centering \scalebox{0.5}{
\includegraphics{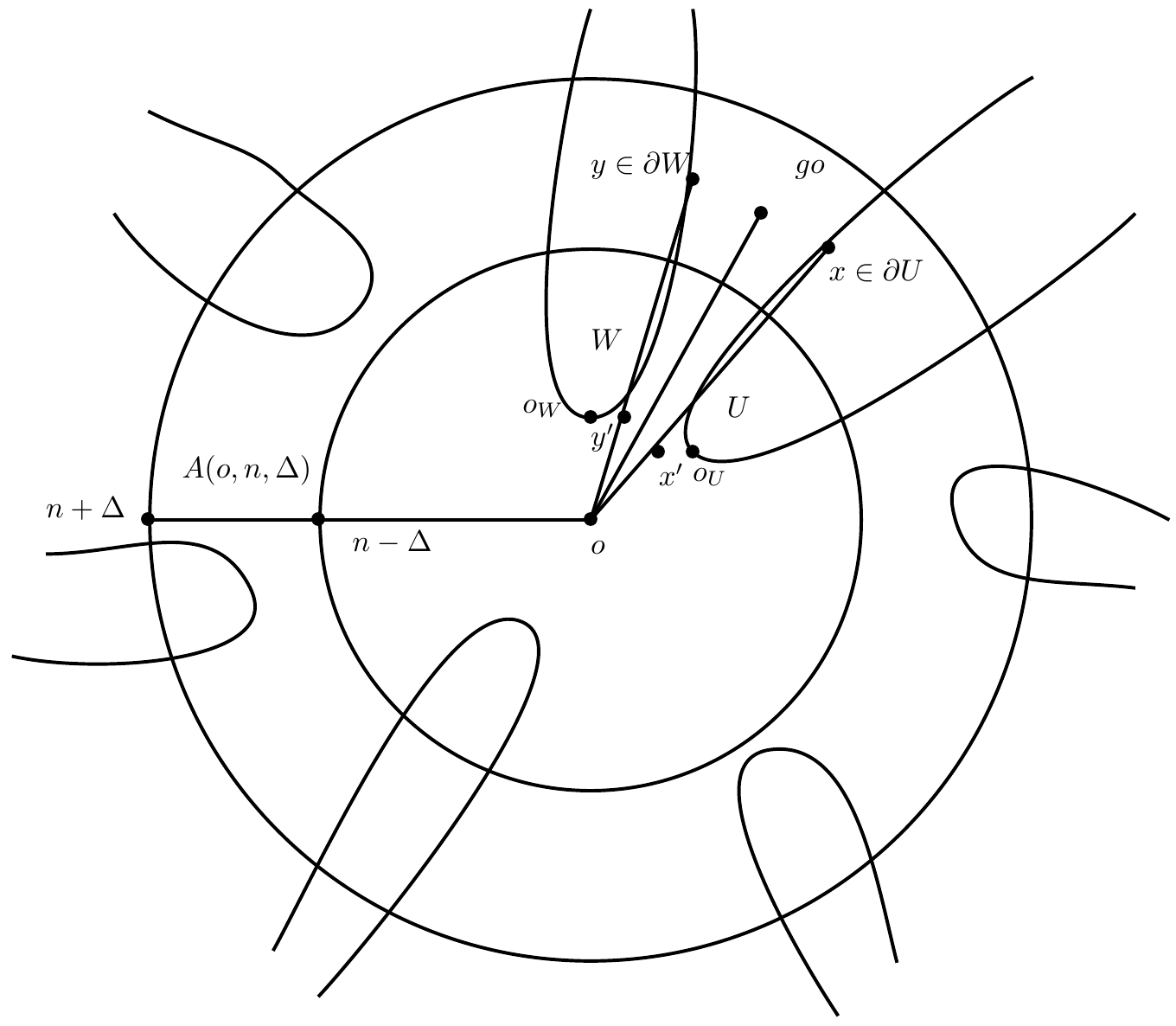} 
} \caption{Sublemma in \ref{orbit=horoball}} \label{Figure1}
\end{figure}

\textbf{Step 2}. By definition,  the set $\mathbb B$ consists of elements $g \in A(o, n, \Delta)$ with $d(go, o_U)\ge R$ and $d(go, U)\le M$, where the constant $R>0$ is determined below. We shall decompose $\mathbb B$ into layers $B_i(o, n, \Delta).$ We first determine the range, the upper bound, of indexes $i$.

 By (\ref{ougo1EQ}), we have 
\begin{equation} \label{ougoEQ}
\begin{array}{ll}
|d(o_U, go)-(d(o, go)-d(o, o_U))|&\le  2d(o_U, [o, go]) \\
&\le 2(M+\delta).
\end{array}
\end{equation}
It follows from  $g\in A(o, n, \Delta)$ that   
$$
\begin{array}{ll}
d(o, o_U)  &\le d(o, go)-d(o_U, go)+2(M+\delta)\\
&\le n -R +\Delta+2(M+\delta).
\end{array}
$$

We thus decompose the set $\mathbb B$ as the union of \textit{layers} $B_i(o, n, \Delta)$ for 
\begin{equation} \label{nilowbdEQ}
0\le i\le n-R+\Delta+2(M+\delta).
\end{equation}
Namely, we define $B_i(o, n, \Delta)$ to be  the set of $g \in \mathbb B$ such that
\begin{equation}\label{BiEQ}
-\Delta\le d(o, o_U) - i  < \Delta. 
\end{equation}

The following claim aims to bound $\sharp B_i(o, n, \Delta).$
\begin{claim2}
Set $\Delta_1:=2(\Delta+3M+\delta)$. For each $g \in B_i(o, n, \Delta) $, we have
$$
g\in   A(o, i, \Delta+M) \cdot A_U(t_Uo, n-i, \Delta_1) \cdot  N(o, 2M)
$$
where $U \in G\cdot V$ is uniquely determined by $go$ in the Claim 1, and $t_U\in G$ is given by Lemma \ref{horoball} such that $d(o_U, t_U o)\le M$.
\end{claim2}
\begin{proof}[Proof of the Claim]
By Lemma \ref{horoball} there exists $h\in G_U$ such that  $d(x, ht_Uo) \le M$. 
Observe that 
\begin{enumerate}
\item
$d(ht_Uo, go)\le d(ht_Uo, x)+ d(x, go)\le 2M$ so $(ht_U)^{-1}g  \in N(o, 2M)$,
\item
it follows by (\ref{BiEQ}) that  $t_U \in A(o, i, \Delta+M),$
\item  
since $|d(o, go)-n|\le \Delta$ for $g\in A(o, n, \Delta)$, we have 
$$
\begin{array}{ll}
|d(o, go)-d(o, o_U) -(n-i)| \le 2\Delta,
\end{array}
$$
by (\ref{BiEQ}).
By the triangle inequality, 
$$
\begin{array}{ll}
|d(h\cdot t_Uo, t_Uo)-d(o_U, go)|&\le d( t_Uo, o_U)+ d(go, ht_Uo)\\
& \le 3M.
\end{array}
$$
Together with (\ref{ougoEQ}), the above two inequalities imply: 
$$
\begin{array}{ll}
|d(h\cdot t_Uo, t_Uo)-(n-i)| &\le |d(h\cdot t_Uo, t_Uo)-d(o_U, go)|\\
&\;\;\;\;+ |d(o_U, go)-(d(o, go)-d(o, o_U))| \\
&\;\;\;\;+ |d(o, go)-d(o, o_U) -(n-i)| \\
&\le 2\Delta+5M+2\delta \le   \Delta_1,
\end{array}
$$
i.e.:
$h   \in A_U(t_Uo, n-i, \Delta_1).$ 
\end{enumerate}

In summary, each $g \in \mathbb B$ can be written as $h \cdot t_U\cdot (ht_U)^{-1} g$ for a unique $U\in \mathbb U$.  
\end{proof}

\textbf{Step 3}.  We calculate  $\sharp \mathbb B$ from the sum of $\sharp B_i(o, n, \Delta).$  
By Lemma \ref{biginclusion}, there exists $\Delta_2=\Delta_2(M, \Delta_1)>0$ such that for any $U\in G\cdot V$, we have 
$$
\sharp A_U(t_Uo, n, \Delta_1) \prec \sharp A_V(t_Vo, n, \Delta_2).$$ Since  $\sharp N(o, 2M)$ is uniform,  the above Claim 2  implies   
$$
\sharp B_i(o, n, \Delta) \prec \sharp A(o, i, \Delta+M) \cdot \sum_{V\in\mathbb{\tilde U}} \sharp A_{V}(t_{V} o, n-i, \Delta_2).
$$  
Hence, by (\ref{EQgrowth}), we have
$$
\sharp B_i(o, n, \Delta)\prec  \sum_{U\in\mathbb{\tilde U}}  \frac{\sharp A_V(t_V o, n-i,\Delta_2)}{\exp((n-i)\g G)} \cdot \exp(n\g G).$$
 
By Lemma \ref{convps2}, there exists a constant $R_2>R_1$ such that 
\begin{equation}\label{CSquire}
\sum_{V\in \mathbb{\tilde U}}\left(\sum_{j\ge R_2} \sharp A_V(o_V, j, \Delta_2)\cdot \exp(-\g G j)\right) \le \frac{1}{2C^2}.
\end{equation}

Now we choose a big constant $R>R_2+2(M+\delta)+\Delta$.
The (\ref{nilowbdEQ}) implies
$$
\begin{array}{ll}
n-i \ge R-2(M+\delta)-\Delta>R_2. 
\end{array}
$$
So we can use (\ref{CSquire})  to sum up $\sharp B_i(o, n, \Delta)$ over $i$:
$$
\sharp \mathbb B \prec   \exp(n\g G)\cdot \sum_{V\in\mathbb{\tilde U}} \left(\sum_{j \ge R_2} \frac{\sharp A_V(o_V, j, \Delta_2)}{\exp(j\g G)}\right) \le  \frac{\exp(n\g G)}{2C}.
$$
This is the inequality (\ref{BCardEQ}) what we wanted to prove, so the proof is complete. 
\end{proof}

\section{Shadow Covering Decomposition}\label{Section5}

This section prepares   necessary ingredients in the proof of Theorem \ref{mainthm} in Section \ref{Section6}.  The main structural result is the measure decomposition in Proposition \ref{ApproxShadowWithMeasure} following a shadow covering decomposition in Proposition \ref{ApproxShadow}.  The idea of proof is using the shadows of orbit vertices in the annulus $A(go, n, \Delta)$ to cover $\Pi_{r, \epsilon, R}(go)$. However, the existence of ``solid" horoballs causes a particular difficulty: the annulus around $go$ may not be uniformly spaced by the orbit $Go$. Hence, it is this place where we put effort into the analysis of the distribution of horoballs in cones (cf. Lemma \ref{LayerHoroballs}), and also where the DOP condition takes into action. 

We start by introducing the main technical definition and notations.

\subsection{Annular and horospherical shadows}


\begin{defn}(Horoballs in cones)\label{URGNdefn} 
Let $r, n \ge 0, g\in G$. Denote by $\mathbb U_{r, n}(go)$ the collection 
of horoballs $U\in \mathbb U$ such that $d(o, o_U)\le d(o, go) +n$ and $\gamma\cap U\ne \emptyset$ for a geodesic $\gamma=[o, \xi]$ for some $\xi \in \Pi_{r}(go)$.
\end{defn}

We first fix  some uniform constants throught out this section.
\begin{const}[\textbf{C, M}]\label{Const}
Let $C>0$ be quasiconvexity constant for all $U\in \mathbb U$, and $M \ge \diam{(X \setminus
\mathcal U)/G}$ given by Lemmas  \ref{contracting} and \ref{horoball}. 
\end{const}
 The following useful fact motivates next definitions. 

\begin{lem}\label{deepisdense}
There exist $L_0>0$ and $\tilde r=\tilde r (r)$ with the following property. Let $U\in \mathbb U_{r, n}(go)$. For any $h\in G_U$ such that $d(o_U, ho_U)>L_0$ there exists  a geodesic $\gamma=[o, \xi]$ for $\xi \in \Pi_{\tilde r}(go)$   such that $$d(o_U, z_-),\; d(ho_U, z_+) \le M$$  
where $z_-, z_+$ are the entry and exit points  of $\gamma$ respectively in $N_C(U)$.
\end{lem}
\begin{proof}
First of all, we observe that there exists a geodesic ray $\beta$ with $\beta_-\in \partial U$ such that   $\diam{\proj_U(\beta)}\le M$. Indeed, take an  arbitrary geodesic ray $\beta$ with an endpoint $\xi\in \pX\setminus \Lambda U$. This shows that $\beta$ eventually leaves every neighborhood of $U$. By Lemma \ref{contracting} (1),   a subray outside $N_C(U)$ projects to a set of diameter at most $M$ on $\partial U$.  Connecting a projection point $w'$ to $\xi$ gives a choice of $\beta$ we wanted.   Since $G_U$ acts on $\partial U$ co-compactly, assume for simplicity that the compact fundamental domain is of diameter of at most $M$. Up to a translation, there exists a geodesic ray $\beta$ with $\beta_-\in \partial U$ such that $d(\beta_-, ho_U)\le M$.  

It is easy exercise that the path $\alpha=[o, o_U][o_U, \beta_-]\beta$ is a $L$-local quasi-geodesc for $L:=d(o_U, \beta_-)$: every subpath of length $L$ is a quasi-geodesic with uniform constants depending on $M$ and the quasiconvexity constant $C$ of $U$. It is well-known that if $L\gg 0$, then $\alpha$ is a global quasi-geiodesic. (cf. \cite[Part III, Theorem 1.13]{BriHae}.) Thus,  the stability of quasi-geodesics implies that $\alpha$ stays within a $C$-neighborhood of $[o, \xi]$ (the same constant  $C$ as above to simplify the notations). By Lemma \ref{contracting} (3),   we show that $d(o_U, z_-) \le M$ and $d( {\proj_U(\beta)}, z_+) \le M$. So 
$$d(z_+, ho_U)\le d(z_+, {\proj_U(\beta)})+d(\beta_-, ho_U)\le 2M.$$ 

We now prove that $\xi\in \Pi_{\tilde r}(go)$ for some $\tilde r$. By definition of $U\in \mathbb U_{r, n}(go)$, there exists a geodesic $[o, \zeta]$ for $\zeta\in \Pi_r(go)$ such that $\alpha\cap U\ne \emptyset$. Let $w$ be the entry point of $[o, \zeta]$  in $U$ so we have $d(w, o_U)\le M$ as above. Since $d(w, z_-)\le 2M$, the thin-triangle property implies that $[o, z_-] \subset N_{r_0}([o, w])$ for $r_0=r_0(M)$. Thus, $d(go, \gamma)\le \tilde r=r+r_0$, completing the proof of $\xi\in \Pi_{\tilde r}(go)$.
\end{proof}

To simply the notations in the proof, we use the symbol $p\sim_D q$ to denote the inequality $d(p, q)\le D$, where $p, q\in X$ are two points.  Two constants $d_1 \simeq_D d_2$ if $|d_1-d_2|\le D$.

\begin{defn} (Deep parabolic elements on horospheres)
Let $U \in \mathbb U_{r, n}(go)$, and $L_0, \tilde r=\tilde r (r)$ be given by Lemma \ref{deepisdense}. We consider the set of geodesics $\gamma=[o, \xi]$ for $\xi \in \Pi_{\tilde r}(go)$ with the following property ($\star$):

\begin{itemize}
\item [($\star$)]
the 
point $z\in \gamma$ given by 
\begin{equation}\label{zdefn}
\begin{array}{ll}
d(o, z) -d(o, go) = n
\end{array}
\end{equation}
lies in $N_C(U)$ and 
\begin{equation}\label{zpmdefn}
d(z, \{z_-, z_+\})>L
\end{equation}
where $z_-, z_+$ are the corresponding entry and exit points of $\gamma$ in $N_C(U)$.
\end{itemize} 

For $L\ge L_0,$ the subset $G_{U, L}$     consists of elements $h\in G_U$:
\begin{equation}\label{htUzplus}
\begin{array}{lr}
d(o_U, z_-)\le M, & (z_-\sim_M o_U)\\
d(h\cdot o_U, z_+) \le  M. & (z_+\sim_M ho_U)
\end{array}
\end{equation}
which  exists  by Lemma \ref{deepisdense} for every geodesic $\gamma$ with property ($\star$).
\end{defn}

\begin{rem}
We caution the reader that $G_{U, L}$ depends on crucially on the horoball $U$ and $L>0$. In particular, this set $G_{U, L}$ is not invariant under $G$-tranlation so differs much from the set $G_{gU, L}$ of a horoball $gU$:  see Lemma \ref{GUSeries}.
\end{rem}

In further development, it is useful to mark a reference point $x$ on $\gamma$ for $go$:
\begin{equation}\label{PointX}
d(o, x) = d(o, go) 
\end{equation} 
 so by (\ref{zdefn}), we have $d(x, z)=n$.

For $\gamma=[o, \xi]$ for $\xi \in \Pi_{r}(go)$, we have $d(go, \gamma) \le r$ and so (\ref{PointX}) implies
\begin{equation}\label{GoX} 
\begin{array}{lr}
d(go, x) \le  2r.  & (go\sim_{2r} x)
\end{array}
\end{equation} 
For each $U\in \mathbb U$, there exists an element $t_U \in G$ given by Lemma \ref{horoball}   such that
\begin{equation}\label{oUtUo}
\begin{array}{lr}
d(t_Uo, o_U) \le  M. & (t_Uo\sim_{M} o_U)
\end{array}
\end{equation}

\subsection{Layering Horoballs in Cones}\label{LayerHoroballs}
This subsection introduces some auxiliary sets to   divide $\mathbb U_{r, n}(go)$ into a sequence of annulus sets by the distance $d(go, U)$. 

We first single out  an exceptional set where $o_U$ lies roughly ``below'' $go$.  Precisely, let $\mathbb X$ be the set of horoballs
$U \in \mathbb U_{r, n}(go)$ such that 
\begin{equation}\label{DefnX}
d(o, o_U)- d(o, go)  < r.
\end{equation}
This implies that $$d(go, o_U)\ge r$$ for $U \in \mathbb U_{r, n}(go)\setminus \mathbb X$.

Secondly, for $i, \Delta \ge 0$, consider the annulus set $$\mathbb U_{r}(go, i, \Delta)$$ of
$U\in \mathbb U_{r, n}(go) \setminus \mathbb X$ such that 
the following holds
\begin{equation}\label{anndef}
\begin{array}{ll}
-\Delta \le d(go, o_U) - i < \Delta. & (d(go, o_U)\simeq_{\Delta} i)
\end{array}
\end{equation}

The following corollary  clarifies the the relation $\mathbb U_{r}(o, i, \Delta)$ to the set of horoballs $H(o, i, \Delta)$. 
\begin{cor}\label{SameHoroballCounting}
 Let $g=1$. Then  $\mathbb U_{r}(o, i, \Delta)= H(o, i, \Delta)$ for $i\le  n-\Delta$.
\end{cor}
\begin{proof}
Note that $\Pi_r(o)=\pX$. It is clear that every horoball intersects some geodesic $[o, \xi]$ for $\xi\in \pX$. So, the set   $\mathbb U_{r, n}(o)$ is the same as the set $\mathbb U$ of all horoballs $U\in \mathbb U$ such that $d(o, o_U)\le n$. The conclusion thus holds by definition of $\mathbb U_{r}(o, i, \Delta)$. 
\end{proof}
Now, we make a second group of uniform constants as follows.

\begin{const}[\textbf{r, $\epsilon$, R, $\Delta$, $L_0$}]\label{Const2}
\begin{itemize}
\item
Let $r\ge 2M, \epsilon, R>0$ be given by the partial shadow lemma \ref{PShadowX}. 
\item
Let $\Delta>7r$ statisfy Lemma \ref{ballgrowthII} .
\item
Let $L_0$ given by Lemma \ref{deepisdense}.
\end{itemize}
\end{const}

\begin{lem}\label{UnionSets}
The following hold for any $L>0$:
\begin{equation}\label{UnionSetsEQ}
\displaystyle{\mathbb U_{r, n-L}(go) = \mathbb X \bigcup \left(\bigcup\limits_{i = 0}^{n-L+\Delta} \mathbb U_{r}(go, i, \Delta)\right),}
\end{equation}
where 
\begin{enumerate}
\item
$d(go, U) \le 5r$ for any $U \in \mathbb X$,
\item
$\sharp \mathbb U_{r}(go, i, \Delta) \prec \exp(i\g
G)$ for any $0\le i \le n-L+\Delta$.
\end{enumerate}
\end{lem}

\begin{proof} We first prove the decomposition.  It suffices to prove $d(go, o_U)\le n-L+\Delta$ for $U \in \mathbb U_{r, n}(go) \setminus \mathbb X$.   Recall that $z_-$ is the entry point of $\gamma$ in $U$. We claim that $z_- \in [x, z]_\gamma$. 
Indeed, if not, assume $z_- \in [o, x]_\gamma$ and then $d(o, x) \ge d(o, z_-)$.  This implies
$$
\begin{array}{llr}
d(o, go) &=d(o, x) \ge d(o, z_-)& \\
&\ge d(o, o_U)-d(z_-, o_U)& \\
&\ge d(o,
o_U) - M.& ( \ref{htUzplus}: z_-\sim_M o_U)
\end{array}
$$  By (\ref{DefnX}), we have $d(o,
o_U) - d(o, go) \ge r>M$, giving a contradiction. 

As a consequence of the fact $z_- \in [x, z]_\gamma$, we then obtain  
$$d(x, z_-) \le d(x, z)-d(z_-, z)\le 
n-L$$ where $d(x,z)=n$ and $d(z_-, z)>L$ by (\ref{zpmdefn}). Thus, 
$$
\begin{array}{llr}
d(go, o_U) & \le d(x, z_-) +  d(z_-, o_U) +d(go, x)& \\
&\le d(x, z_-) +  M +2r&  [(\ref{htUzplus}:  z_-\sim_M o_U),\;(\ref{PointX}:  go\sim_{2r} x)]\\
&\le n+3M-L.&
\end{array}
$$
 So   the decomposition (\ref{UnionSetsEQ}) follows.

\textbf{(1)}: Let $U \in \mathbb X$ so that (\ref{DefnX}) holds. We see that $z_-\in N_{2r}([o,
x]_\gamma)$. Indeed, if not, then $z_- \in [x, z]_\gamma$ and $d(x, z_-)
>2r$. As $r > M$, it follows that 
$$
\begin{array}{llr}
d(o, x) &=d(o,z_-)-d(x,z_-)&  \\
&< d(o, o_U)+d(o_U, z_-) -2r&   \\
&< d(o, o_U) + M -2r& (\ref{htUzplus}: z_-\sim_M o_U) \\
&< d(o, o_U) - r,&
\end{array}
$$ giving a contradiction with (\ref{DefnX}). Thus, it is proved that $z_- \in N_{2r}([o,
x]_\gamma)$. 

As $U\in \mathbb U$ is
quasiconvex, let $M$ also satisfy that $[z_-, z]_\gamma \subset
N_M(U)$ for simplifying notations.  
As $z_- \in N_{2r}([o,
x]_\gamma)$, we obtain that
$d(x, U) \le M+2r$. So by (\ref{GoX}) we have $d(go, U) \le M +4r \le 5r$ for $U \in \mathbb X$. The statement \textbf{(1)} is proved.

\textbf{(2)}:   This follows from Lemma \ref{upperhgrowth}.
\end{proof}

The following result is a consequence of Lemma \ref{deepisdense}. 
 
\begin{lem}\label{GUSeries}
For $L>L_0$,  
$$
\mathcal A_U(o_U,   d(o_U, z)+2L, \Delta) \prec \sum_{h\in G_{U, L}} \exp(-d(o_U, ho_U)\g G) \prec
  \mathcal A_U(o_U, d(o_U, z)+L, \Delta). 
$$
\end{lem}

\begin{proof}
Recall that the proof of Lemma \ref{UnionSets} (1) shows that $z_- \in [x, z]_\gamma$, and $d(z_-, z_+) = d(z_-, z)+d(z, z_+)$.   By $(\ref{htUzplus}: o_U\sim_M z_-)$ and $(\ref{htUzplus}: ho_U\sim_M z_+)$, we obtain 
$$
\begin{array}{rl}
d(o_U, h\cdot o_U) \simeq_{3M}  d(z, z_+) +d(o_U, z). 
\end{array}
$$
Since $d(z, z_+) > L$ by (\ref{zpmdefn}) and $r>M$, we obtain
$$
\begin{array}{rl}
d(o_U, h\cdot o_U)&\ge d(o_U, z) +d(z, z_+) -3M \\
&\ge    d(o_U, z)+ L -3r.
\end{array}
$$   
This shows the upper bound:
$$
\begin{array}{rl}
\sum_{h\in G_{U, L}} \exp(-d(o_U, ho_U)\g G) &\le \mathcal S_U(o_U, o_U, d(o_U, z)+L-3r)\\
&\le \mathcal A_U(o_U, d(o_U, z)+L, 3r).
\end{array}
$$

For the lower bound, let $h\in G_U$ such that $d(o_U, h\cdot o_U)\ge    d(o_U, z)+ 2L+3r$. Since $L>L_0$ and $d(o_U, h\cdot o_U)>L_0$, By Lemma \ref{deepisdense}, there exists a geodesic $\gamma=[o, \xi]$ for $\xi \in \Pi_{\tilde r}(go)$ such that $d(z_-, o_U)\le M$ and $d(z_+, ho_U)\le M$. 

Noting that $d(o_U, h\cdot o_U)>2L+3r$, we have $d(z_-, z_+)\ge d(o_U, h\cdot o_U)-2M>2L$ and there exists $z \in \gamma$ given by (\ref{zdefn}) such that $d(z, \{z_-, z_+\})>L$. This proved $h \in G_{U, L}$, so gives the lower bound:
$$
\begin{array}{rl}
\sum_{h\in G_{U, L}} \exp(-d(o_U, ho_U)\g G) &\ge \mathcal S_U(o_U, o_U, d(o_U, z)+2L+3r)\\
&\succ \mathcal A_U(o_U, d(o_U, z)+2L, 3r),
\end{array}
$$
where the last inequality follows from Lemma \ref{ConvertSA} (1). 
\end{proof}

\subsection{Shadow Covering Decomposition}
Recall that $\Pi^c_{r, \epsilon, R}(go)$ denotes the set of conical
limit points in $\Pi_{r, \epsilon, R}(go)$, and the $\Pi_{r, \epsilon, R}(go)$ is defined relative to a system $\mathbb Y$ of quasiconvex subsets in Convention \ref{ConvContracting}.

\begin{prop}\label{ApproxShadow}
There exist $\tilde R=\tilde R(\epsilon, R), \tilde r=\tilde r(r), \Delta >0$ with the following property.

For any $L>L_0$, there exists $\tilde\Delta=\tilde\Delta(L) >0$ such that the set $\Pi^c_{r, \epsilon, R}(go)$ is covered by  the following three collections of shadows:
\begin{enumerate} 
\item
the \textbf{annular} shadows: 
$$
\Pi_1:=\{ \Pi_{r}(ho): h \in
\Omega_{\tilde r, \epsilon, \tilde R}(go, n, \tilde\Delta)\},  
$$
\item
the \textbf{horospherical} shadows:
$$
\Pi_2:=\bigcup_{0\le i\le n-L+\Delta} \left(\bigcup_{U \in \mathbb U_{r}(go, i, \Delta)} \mathbf {Sha}_{L}(U)\right)\; 
$$
\item
the exceptional \textbf{horospherical} shadows:
$$
\Pi_3:=\bigcup_{U \in \mathbb X} \mathbf {Sha}_{L}(U). 
$$
where $\mathbf {Sha}_L(U):=  \{\Pi_r(h\cdot o_U): h \in G_{U, L}\}$.
\end{enumerate}
whose union $\Pi_1\cup \Pi_2\cup \Pi_3$ is, in turn, contained in $\Pi_{2\tilde r}(go)$ for any $n \gg 0$.  
\end{prop}

\begin{proof}
\textbf{(1).} Set $\tilde\Delta = L + 2M$. We prove first that $\Pi^c_{r, \epsilon, R}(go)$ is covered by $\Pi_1\cup \Pi_2\cup\Pi_3$. For any $\xi \in \Pi^c_{r, \epsilon, R}(go)$, there exists a
geodesic $\gamma=[o,\xi]$ such that $\gamma$ contains an $(\epsilon,
R)$-transition point $v$ in $B(go, 2R)$.  We have two cases to consider as follows.

\textbf{Case 1.}  There exists $h \in G$ such that $d(ho, z) \le M$. Since it is assumed that $M \le r$, it follows that 
\begin{equation}\label{E1}
\xi \in   \Pi_r(ho).
\end{equation}
In this case, it remains to prove that 
$h \in
\Omega_{\tilde r, \epsilon, \tilde R}(go, n, \tilde\Delta)$ for constants $\tilde r, \tilde R$ to be determined below.

Let $L_1=L_1(\epsilon, R, M), D=D(\epsilon, R)$ be given by Lemma
\ref{PairTransX}. Choose $$n>\max\{L_1+2R, M+r\}.$$

We first prove that $d(go, [o, ho]) \le \tilde r$ for $\tilde r:=r+\delta>0$. Indeed,  since $\gamma\cap B(go, r)\ne \emptyset$ for $\xi\in \Pi_r(go)$, there exists $\bar z\in \gamma$ such that $d(\bar z, go)\le r$. Noting that $d(o, z) -d(o, go) =n$ in (\ref{zdefn}) and $n\gg  2r$, we see $\bar z\in [o, z]_\gamma$ and $d(\bar z, z)>r$.   Note that two geodesics $[o, z]_\gamma$ and $[o, ho]$ have two endpoints with a distance at most $r>M\ge d(z, ho)$.  By Lemma \ref{thintri}, we see that $d(\bar z, [o, ho])\le \delta$.  Thus 
\begin{equation}\label{E1prime}
d(go, [o, ho])\le d(go, \bar z)+d(\bar z, [o, ho])\le \tilde r.
\end{equation}
 

Since $d(z, ho)\le M$ and $d(o, z)-d(o, go)=n$ (\ref{zdefn}), we deduce   that $$|d(ho, go)-n|\le M \le \tilde \Delta,$$   which  proves that $h\in \Omega_{\tilde r}(go, n,\tilde \Delta)$.
 
We now prove that $[o, ho]$ contains an $(\epsilon, R)$-transition point in $B(go, 2R+D)$. Recall that $\gamma$ contains a transition point $v$ in $B(go, R)$ so $d(v, go)\le R$. By triangle inequality, we have $d(z, go)+d(o, go)+2d(v, go)\ge d(o, z)$.  Hence, we obtain 
$$
\begin{array}{llr}
d(z, v)&\ge d(z, go)-d(go, v) &\\
&\ge d(o, z)-d(o, go)-3d(go, v)&\\
& \ge n-3R>L_1. & \;\;[(\ref{zdefn}):\;d(v, go)\le R] 
\end{array}
$$
  By Lemma \ref{PairTransX},
there exists an $(\epsilon, R)$-transition point $w$ in $[o, ho]$
such that $d(v, w) < D$ so $d(w, go) < 2R+D$. Denoting $\tilde R=2R+D$, we have proved 
\begin{equation}\label{E2}
h \in \Omega_{\tilde r, \epsilon, \tilde R}(go, n, \tilde\Delta).
\end{equation}
In this case, we proved that
$\xi\in \Pi^c_{r, \epsilon, R}(go)$ lies in a shadow of annular type in $\Pi_1$.

\begin{figure}[htb]
\centering \scalebox{0.5}{
\includegraphics{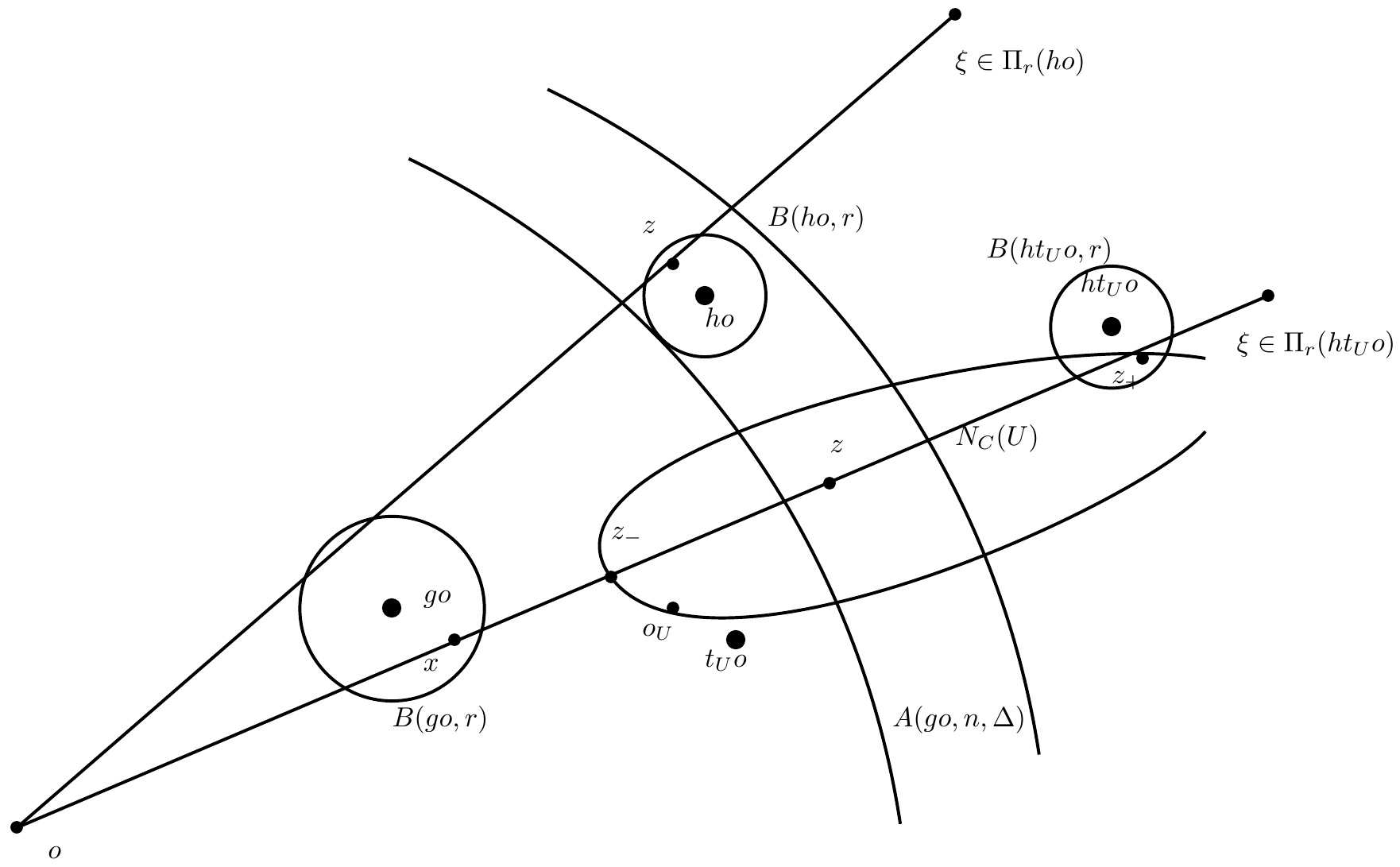} 
} \caption{Proposition \ref{ApproxShadow}} \label{Figure2}
\end{figure}
\textbf{Case 2}. The ball $B(z, M)$ contains no point in $Go$. Since $M \ge \diam{(X \setminus
\mathcal U)/G}$, there exists a horoball $U \in \mathbb
U$ contains the point $z$. Let $z_-, z_+$ be the corresponding entry and exit points of $\gamma$ in $N_C(U)$. Note that $\partial U \subset N_M(G_U\cdot o_U)$ so there exists $h\in G_U$ such that 
\begin{equation}\label{zpmEQ}
d(z_-, o_U),\; d(z_+,
ho_U) \le M,
\end{equation} 
where the first inequality follows by Lemma \ref{contracting}.

To finish the discussion of \textbf{Case 2}, we continue to examine two subcases  
 as follows.  

\textit{Subcase 1}: $\min\{d(z, z_-), d(z, z_+)\}\le L$. Thus,  by (\ref{oUtUo}) and (\ref{zpmEQ}), we have $$d(z, \{t_Uo, ht_Uo\}) \le L+2M\le \tilde \Delta.$$ For definiteness, consider $d(z,   ht_Uo) \le \tilde \Delta$, so two geodesics $[o, z]$ and $[o, ht_Uo]$ with two endpoints at most $\tilde \Delta$. For $n\gg 2\tilde \Delta$, we argue by $\delta$-thin triangle property as   proving  (\ref{E1prime}) to show that $d(go, [o, ht_Uo])   \le  \tilde r$. Thus, $$ht_U\in \Omega_{\tilde r}(go, n, \Delta).$$   

Moreover, repeating the argument in the \textbf{Case 1}, we see that there exists an $(\epsilon,
R)$-transition point $w$ in $[o, h\cdot t_Uo]$ such that $d(go, w) <
2R+D$. In a word, we proved that
\begin{equation}\label{E4}
ht_U \in \Omega_{\tilde r, \epsilon, \tilde R}(go, n, \tilde\Delta).
\end{equation}

Hence, $\xi\in \Pi^c_{r, \epsilon, R}(go)$   goes into a shadow of annular type in $\Pi_1$ as well.

\textit{Subcase 2}: $\min\{d(z, z_-), d(z, z_+)\}> L$. As $\xi$ is a conical point, any geodesic $[o, \xi]$ leaves every
horoball into which it enters. This subscase is not vacuous. 

Since $d(z_-, z)>L$ and $d(x, z)=n$ by (\ref{zdefn}), we have $d(z_-, x)<n-L$.  By (\ref{PointX}), we have  $d(o, x)=d(o, go)$. It follows that   $$
\begin{array}{llr}
d(o, o_U)-d(o, go)&\le d(o, z_-)+d(z_-, o_U) -d(o, x)&\\
& \le d(o, z_-) -d(o, x)+M&  \;\; [(\ref{zpmEQ}): z_-\sim_M o_U] \\
&\le d(z_-, x)+M\le n-L+M.&
\end{array}$$
By definition, this shows that $U\in \mathbb U_{r, n-L+\Delta}$. By (\ref{zpmEQ}), the element $h$ satisfies $d(h o_U, z_+) \le M$ and $d(z_-, z_+)>L$ so by definition, $h\in G_{U, L}$.

Note that $$d(ht_Uo, z_+)\le d(ht_Uo, ho_U)+d(ht_Uo, z_+)\le 2M\le r$$ by (\ref{zpmEQ}, \ref{oUtUo}), yielding
\begin{equation}\label{E3}
\xi \in  \Pi_r(h\cdot t_Uo) 
\end{equation}
thus completing the proof that $\xi\in \Pi^c_{r, \epsilon, R}(go)$ lies in a shadow of horospherical type in $\Pi_2\cup \Pi_3$. 

Therefore, we proved that $\Pi^c_{r, \epsilon, R}(go)$ is contained in the union of collections  $\Pi_1 \cup \Pi_2\cup \Pi_3$ of shadows.  

\textbf{(2).} We now show that the union of all shadows from $\Pi_1\cup \Pi_2\cup \Pi_3$ is contained in $\Pi_{2\tilde r}(go)$. For any $\Pi_{r}(ho) \in \Pi_1$, we prove that $\Pi_{r}(ho)\subset \Pi_{2\tilde r}(go)$. The cases for $\Pi_2\cup \Pi_3$ are similar. 

Since $ho \in \Omega_{\tilde r, \epsilon, \tilde R}(go, n, \tilde\Delta)$, we have $d(go, [o, ho])\le \tilde r$. Let $\xi \in \Pi_{r}(ho)$ so we have $d(ho, [o, \xi])\le r$. By the $\delta$-thin triangle property, we see that $d(go, [o, \xi]\le \tilde  r$ by the same argument for (\ref{E1prime}).  Hence,   $\Pi_{r}(ho) \subset  \Pi_{2\tilde r}(go)$.  This concludes the proof of the proposition.
\end{proof}

\subsection{Measure decomposition}

Following Proposition \ref{ApproxShadow}, the goal of this section is to prove the following.
\begin{prop}[Measure decomposition]\label{ApproxShadowWithMeasure}
Under the same assumption as Proposition \ref{ApproxShadow}, we have
\begin{equation}\label{F1}
\begin{array}{ll}
\exp( -\g G  d(o, go))
\asymp  
 \mathbf {ASha}(g, n,\tilde \Delta) + \mathbf {HSha}(g, n, L)+\mathbf {ESha}(g, n, L),
\end{array}
\end{equation}
where 
$$
\mathbf {ASha}(g, n,\tilde \Delta):=\sum\limits_{\Pi \in \Pi_1} \mu_1(\Pi), \;
\mathbf {HSha}(g, n, L):=\sum\limits_{\Pi \in \Pi_2} \mu_1(\Pi), \;
\mathbf {ESha}(g, n, L):=\sum\limits_{\Pi \in \Pi_3} \mu_1(\Pi).
$$  
\end{prop}
\begin{rem}
Each of the above sums   depends on the important parameter $L>0$, which shall be tweaked so that $\mathbf {HSha} $ and $\mathbf {ESha}$ become arbitrarily small in Section \ref{Section6}.  The remainder, subSections \ref{MeasHSha} and \ref{MeasESha}, of this section is to setup necessary estimates on  $\mathbf {HSha} $ and $\mathbf {ESha}$ to implement this goal.
\end{rem}

In the proof, we need the following convexity property for a horoball $U\in \mathbb U$:  
\begin{lem}\cite[Lemma 2.3]{Hru}\label{horoballconvexity}
Let  $\gamma$ be a geodesic with endpoints  $\gamma_-, \gamma_+ \in \partial U$. For any $r>0$, there exists  $K=K(r)$ such that $N_r(\partial U)\cap \gamma\subset N_K(\{\gamma_-, \gamma_+\})$. 
\end{lem}
\begin{proof}[Proof of Proposition \ref{ApproxShadowWithMeasure}]
Since $G$ is of divergent type, we have that $\mu_1$ has no atoms at parabolic points by Lemma \ref{parabnoatom2}. So $\mu_1(\Pi^c_{r, \epsilon, R}(go))=\mu_1(\Pi_{r, \epsilon, R}(go))$.  By the partial shadow Lemmas \ref{PShadowX} and  \ref{OmbreLem}, we have 
$$
\mu_1(\Pi_{r, \epsilon, R}(go))\asymp \exp( -\g G (d(o, go)) \asymp \mu_1(\Pi_{2\tilde r}(go)).
$$
The direction ``$\prec$" follows from Proposition \ref{ApproxShadow}.   For the direction ``$\succ_r$", it sufices  to establish that any $\xi\in \pX$ is  evenly covered by shadows from  $\Pi_1$ and $\Pi_2\cup \Pi_3.$

The evenly covering property of the  collection $\Pi_1=\{\Pi_r(go): g\in A(o, n, \tilde\Delta)\}$ is known  in \cite[Lemme 6.5]{Coor}:
\begin{claim} \label{evenlycovered}
For any $r>0$, there exists $N=N(r, \tilde \Delta)>0$ such that for any $n\ge 0$ a point $\xi\in \pX$ is covered at most $N$ shadows from the collection $\{\Pi_r(go): g\in A(o, n, \tilde\Delta)\}$.
\end{claim}
The remaining of the proof is to prove the above evenly covering property for
$$\Pi_2\cup \Pi_3=\{\Pi_r(h\cdot t_Uo):h \in G_{U, L}, U \in \mathbb U_{r,
n}(go) \}.$$   
Consider one shadow $\Pi_r(h\cdot t_Uo)$  which contains a fixed point $\xi \in \pX$. Recall that $\partial U \subset N_M(G_U\cdot t_Uo)$ by Lemma \ref{horoball}.  Since $\xi \in \Pi_r(h\cdot t_Uo)$, there exists $w \in [1, \xi]$ such that $d(w, h\cdot t_Uo) \le r$ and then $d(w, \partial U) \le r+M$. 

Let $z_-, z_+$ be the corresponding entry and exit points of $[1, \xi]$ in $U$.  By the above convexity of horoballs, we have 
\begin{equation}\label{zpmwEQ}
\min\{d(z_-, w), d(z_+, w)\} \le K
\end{equation}
where $K:=K(r+M)$ is given by Lemma \ref{horoballconvexity}. Observe that $d(z_-, w) > K$. Indeed, if not, we have $d(z_-, w)\le K$ and then $d(z_-, h\cdot t_Uo)\le r+K$. By definition of $G_{U,L}$, we have    $d(h\cdot t_Uo, z_+) \le  M$ in (\ref{htUzplus}) so    $d(z_-, z_+)\le M+K+r$. On the other hand,  by (\ref{zpmdefn}), we have $d(z_-, z_+)>2L$. We got a contradiction by assuming that $$L>M+K+r.$$   Thus, it follows by (\ref{zpmwEQ}) that $d(z_+, w)\le K$, and $d(z_+, h\cdot t_Uo)\le d(z, w)+d(w, h\cdot t_Uo)\le r+K.$ Hence, the proper action implies that there is at most a uniform number of $\Pi_r(h\cdot t_Uo)$ containing $\xi$. The proof is complete.
\end{proof}

\subsection{Measuring horospherical  shadows}\label{MeasHSha}
This subsection aims to estimate the sum  $\mathbf {HSha}(g, n, L)$  in (\ref{F1})  which measures the union of shadows $\Pi_r(h\cdot t_Uo)$ on each   horoball  $U \in \mathbb U_{r, n}(go, i, \Delta)$ where $0\le i \le n-L+\Delta$. We first sum up the ones from one single horoball. 

\begin{lem}[Single horoball in $i$-th annulus]\label{SUexpdecay}
Under the same assumption as Proposition \ref{ApproxShadow}. Given $U \in \mathbb U_{r}(go, i, \Delta)$ for $i\le n-L+\Delta$, we have   
\begin{equation}\label{SUepsilonDOP}
\begin{array}{cc}
   \exp( -\g G (d(o, go) +i)) \cdot \mathcal{A}_{U}(o_U, n-i+2L, \Delta)\\
   \\
 \prec  \sum\limits_{\Pi \in \mathbf{Sha}_L(U)} \mu_1(\Pi) \prec \\
 \\
     \exp( -\g G (d(o, go) +i)) \cdot \mathcal{A}_{U}(o_U, n-i+L, \Delta).
\end{array}
\end{equation}
 
\end{lem}

\begin{proof}
Recall that the proof of Lemma \ref{UnionSets} (1) shows that $z_- \in [x, z]_\gamma$, and thus  $d(o, z_+)  = d(o, z)-d(z_-, z)+d(z_-, z_+)$. By (\ref{htUzplus}), we have $z_+\sim_M h\cdot o_U$ and $z_-\sim_M o_U$. By (\ref{zdefn}), we obtain  
\begin{equation}\label{EQohtuo}
\begin{array}{rl}
d(o, h\cdot o_U)  &\simeq_{4M} \big(d(o, go) + n\big) + d(o_U, h\cdot o_U)  - d(o_U, z).
\end{array}
\end{equation}

Since $d(x, z)=n$ and $z_- \in [x, z]_\gamma$, we have $n=d(x, z_-)+d(z_-,z)$.  Thus, 
$$
\begin{array}{rl}
 d(o_U, z) +d(o_U, go) \simeq_{4r} n. &\; [(\ref{GoX}): go\sim_{2r} x]
\end{array} 
$$
which in turn implies 
\begin{equation}\label{EQoUz}
d(o_U, z) \simeq_{\Delta+4r} n-i.
\end{equation}
where by (\ref{anndef}), $d(go, o_U)\simeq_\Delta i$ for $U \in \mathbb U_{r, n}(go, i, \Delta)$.
 
By (\ref{EQohtuo}) and (\ref{EQoUz}), we have
$$
\begin{array}{rl}
&\exp(-\g G d(o, h\cdot o_U)) \\
 \asymp_{r, \Delta}& \exp(-\g G (d(o, go)+i) ) \cdot \exp(-\g G d(o_U,
h\cdot o_U)).
\end{array}
$$

As a consequence, 
$$
\begin{array}{lll}
\sum\limits_{h\in G_{U, L}}  \exp(-\g G d(o, h\cdot t_Uo)) &  \asymp_{r, \Delta  }   \exp(-\g G (d(o, go) +i))  \cdot  \\
&\;\;\left(  \sum\limits_{h\in G_{U, L}}  \exp(-\g G d(o_U, h\cdot t_Uo)) \right)
\end{array}
$$
where      $t_Uo \sim_M o_U$ by (\ref{oUtUo}) is used. 

The conclusion follows from  the shadow lemma \ref{OmbreLem} and Lemma \ref{GUSeries}.
\end{proof}

Recall that $\mathbb{\tilde U}:=\{U_k \in \mathbb U: 1\le k\le m\}$ is a choice of representatives in each $G$-orbit.

\begin{lem}[All Horoballs] \label{Uexpdecay} 
 Under the same assumption as Proposition \ref{ApproxShadow},
\begin{enumerate}
\item
The following holds for   any $L\gg L_0$ and $n\gg 0$:
\begin{equation}\label{UepsilonDOP}
\begin{array}{rl}
 \mathbf {HSha}(g, n, L) \prec \exp( -\g G  d(o, go) ) \cdot  \sum\limits_{V\in \mathbb{\tilde U}}  \left(\sum\limits_{j\ge 2L} \mathcal{A}_{V}(o_V, j, \Delta)\right)
\end{array}
\end{equation}

\item 
Assume that $G$ has purely exponential horoball growth.  Then
\begin{equation}\label{LepsilonDOP}
\begin{array}{rl}
\mathbf {HSha}(1, n, L)\succ   \sum\limits_{V\in \mathbb{\tilde U}}  \Big(\sum\limits_{j\ge 2L}^n \mathcal{A}_{V}(o_V, j, \Delta)\Big).
\end{array}
\end{equation}
\end{enumerate}
\end{lem}

\begin{proof}
\textbf{(1).} By Lemma \ref{UnionSets} (2), we have
\begin{equation}\label{sumhoroball}
\sharp \mathbb U_{r}(go, i, \Delta) \prec \exp(i\g G).
\end{equation}
for any $i\le n-L+\Delta$. Moreover, for any $U=gV \in \mathbb U_{r}(go, i, \Delta)$, we have
\begin{equation}\label{SeriesConjugate}
\mathcal{A}_{V}(o_V, n-i+2L+K, \Delta)\prec_\Delta \mathcal{A}_{U}(o_U, n-i+L, \Delta) \prec_\Delta \mathcal{A}_{V}(o_V, n-i+L-K, \Delta)
\end{equation}
where the  constant $K$ is given by Lemma \ref{biginclusion}.

Set $\Delta>K$. Now taking into account (\ref{sumhoroball}) and (\ref{SeriesConjugate}), we sum up (\ref{SUepsilonDOP}) over $0<i\le n-L+\Delta$ to get  
 
$$
\begin{array}{rl}
 \mathbf {HSha}(g, n, L) \prec  & \sum\limits_{i= 0}^{n-L+\Delta}  \left( \sum\limits_{V\in \mathbb{\tilde U}} \mathcal{A}_{V}(o_V, n-i+L-K, \Delta) \right)\\
\\
\prec  &\sum\limits_{j\ge 2L-\Delta-K}^{n+L-K}\left(\sum\limits_{V\in \mathbb{\tilde U}}   \mathcal{A}_{V}(o_V, j,\Delta) \right) \\
\\
\prec & \sum\limits_{V\in \mathbb{\tilde U}}  \Big(\sum\limits_{j\ge L}  \mathcal{A}_{V}(o_V, j,\Delta)\Big).

\end{array}
$$
The statement (1) is proved.

\textbf{(2).} Assume now that the horoball growth function is purely exponential. By Lemma \ref{orbit=horoball},  
\begin{equation}\label{sumhoroball2}
\exp(i\g G) \prec \sharp H_V(o, i, \Delta).
\end{equation}
for each $V\in \mathbb U$. Thus, by  Lemma \ref{SameHoroballCounting}, $H_V(o, i, \Delta)\asymp \mathbb U_{r}(o, i, \Delta)$ for $i<n-L+\Delta$.

For $L>\Delta,$  the (\ref{SeriesConjugate}) implies
$$
\begin{array}{rl}
 \mathbf {HSha}(1, n, L) \succ  & \sum\limits_{i= 0}^{n-L+\Delta}  \left( \sum\limits_{V\in \mathbb{\tilde U}} \Big(\mathcal{A}_{V}(o_V, n-i+2L+K, \Delta ) \Big) \right)\\
\\
\succ  &\sum\limits_{j\ge 3L-\Delta+K}^{n+2L+K} \left(   \sum\limits_{V\in \mathbb{\tilde U}}\Big(\mathcal{A}_{V}(o_V, j, \Delta ) \Big) \right)\\
\\
\succ & \sum\limits_{V\in \mathbb{\tilde U}}  \Big(\sum\limits_{j\ge 2L}^n \mathcal{A}_{V}(o_V, j, \Delta)\Big),

\end{array}
$$
proving the statement (2).
\end{proof}

\subsection{Exceptional horospherical shadows}\label{MeasESha}
We close this section by estimating the last piece, $\mathbf{ESha}(g, n, L)$, which is the measure of shadows from ``exceptional'' horoballs in $\mathbb X$.

\begin{lem}\label{farhoroball}
 For any $\varepsilon>0$ there exists $L_2=L_2(\epsilon)>0$ such that the
following holds
\begin{equation}\label{F2}
\mathbf{ESha}(g, n, L)\prec  \exp(-\g G
d(o, go))\cdot \varepsilon
\end{equation}
for any $n\gg 0$ and $L>L_2$. 
 
\end{lem}
\begin{proof}
By $(\ref{htUzplus}: ho_U\sim_M z_+)$ and $(\ref{oUtUo}: o_U\sim_M t_Uo)$, we have $(ht_Uo\sim_{2M} z_+)$.
Since $d(o, z_+) = d(o, x) + d(x, z_+)$ and $d(o, go)=d(o, x)$, we have 
\begin{equation}\label{ohtuo}
\begin{array}{lr}
d(o, h\cdot t_Uo) \simeq_{6r} d(o, go) + d(go, h\cdot t_Uo). &  (\ref{GoX}: go\sim_{2r} x) 
 
\end{array}
\end{equation}

Since $d(go, U) < 5r$ for $U\in \mathbb X$, we have that $\mathbb X$ is a finite set.   By Lemma \ref{convps2} we have 
\begin{equation}\label{GUsum}
\sum\limits_{h \in G_U} \exp(-\g Gd(v, hw)) <   \infty
\end{equation}
where $v, w\in N_{5r}(U)$. We apply (\ref{GUsum})   for $v:=go, w:=t_Uo$.  

Noting that $$d(go, h\cdot t_Uo)\ge d(x, z_+)-d(x, go)-d(z_+, ht_Uo)\ge n+L-3r,$$ so by (\ref{ohtuo}) we have
$$
\begin{array}{rl}
\mathbf{ESha}(g, n, L)\prec&\sum\limits_{U \in \mathbb X}\left(\sum\limits_{h \in G_U} \exp(-\g G d(o, go)) \cdot \exp(-\g Gd(go, h\cdot t_Uo))\right) \\
\prec &  \exp(-\g G d(o, go))) \cdot \left(\sum\limits_{U \in \mathbb X}\sum\limits_{h \in G_U}^{d(go, h\cdot t_Uo)\ge n+L-3r}\exp(-\g Gd(z, hw))\right).
\end{array}
$$
Hence, the lemma follows from the convergence of  (\ref{GUsum}).  
\end{proof}

\section{Proof of Main Theorem}\label{Section6}
Recall the goal of the paper is to prove
\begin{thm}
Suppose $G$ admits a cusp-uniform action on a proper hyperbolic space $(X,d)$ such that $0<\g G <\infty$ and $G$ is of divergent type. Then the following statements are equivalent:
\begin{enumerate}
\item
$G$ satisfies the DOP condition.  
\item
$G$ has the purely exponential orbit growth.
\item
$G$ has the purely exponential orbit growth in (partial) cones.
\item
$G$ has the purely exponential horoball growth.  
\end{enumerate}
\end{thm}

The directions $(2)\Leftrightarrow (4)$ are already proved in Section \ref{Section4}, and   it is trivial that $(3)\Rightarrow (2)$. So it remains to show $(1)\Rightarrow (3)$ and $(4)\Rightarrow (1)$. 
Lets first explain a characterization  of the DOP condition.  

\begin{lem}[DOP condition]\label{DOP}
The group $G$ satisfies the DOP condition if and only if each series 
$$
\sum\limits_{p\in G_U} d(v, pv) \cdot \exp(-\g G d(v, pv)) \asymp_\Delta \sum_{j\ge 0} \sum\limits_{m\ge j}  \sharp A_U(v, m, \Delta) \cdot \exp(-\g G m)
$$
is convergent for any $U \in \mathbb U, v \in \partial U$ and $\Delta>1$.
\end{lem}
\begin{proof}
Observe that
$$
\sum\limits_{p\in G_U} d(v, pv) \cdot \exp(-\g G d(v, pv)) \asymp_\Delta \sum\limits_{m\ge 0} m\cdot \sharp A_U(v, m, \Delta) \cdot \exp(-\g G m),
$$
for any $\Delta \ge 1$. Indeed, the ``$\le$" direction is trivial. For the ``$\succ_\Delta$" direction, it suffices to notice that for $p\in G_U$ the point $pv$ lies in at most $2\Delta+1$ sets from $\{A_U(v, m, \Delta): m\in \mathbb N\}$.   

Denoting $a_m=\sharp A_U(v, m, \Delta) \cdot \exp(-\g G m)$, we have 
$$\sum_{m\ge 0} m\cdot a_m = \sum_{j\ge 0}\sum_{m\ge j} a_m.$$
The conclusion thus follows.
\end{proof}

 In terms of partial sum $\mathcal A_U(o, R, \Delta)$ (\ref{SeriesDelta}), the series in the lemma above can be repharased as 
$$
\sum\limits_{p\in G_U} d(v, pv) \cdot \exp(-\g G d(v, pv)) \asymp_\Delta \sum_{j\ge 0}   \sharp \mathcal A_U(v, j, \Delta),
$$
for each $U\in \mathbb U$.

\subsection{\textbf{Direction $(1)\Rightarrow (3)$}:} 
Assuming the DOP condition, we prove the purely exponential orbit growth in partial cones. Let $\mathbb U\subset \mathbb Y$ be any system of quasiconvex subsets in Convention \ref{ConvContracting}. For any $g \in G$, we shall prove that for sufficiently large  $\tilde \Delta$, the partial cone $\Omega_{\tilde r, \epsilon,\tilde R}(go, n, \tilde\Delta)$ is purely expeonential, where $\tilde r, \tilde R$ are given by Proposition \ref{ApproxShadow}. 
  
For any $L>L_0$,   there exist constants $C_1>0$ and $\tilde \Delta=\tilde \Delta(L) >0$ by Proposition \ref{ApproxShadowWithMeasure} such that the following holds
$$
\begin{array}{ll}
C_1\exp( -\g G d(o, go)) \le 
 \mathbf {ASha}(g, n, \tilde \Delta) + \mathbf {HSha}(g, n, L)+\mathbf {ESha}(g, n, L),
\end{array}
$$
where $$\mathbf {ASha}(g,n,\tilde \Delta)=\sum \mu_1(\Pi_r(go))$$ for  $g\in \Omega_{\tilde r, \epsilon, \tilde R}(go, n, \Delta)$

For $L\gg 0$, the DOP condition implies $$\left(\sum\limits_{j\ge 2L} \mathcal{A}_{V}(o_V, j, \Delta)\right) \to 0,$$ by Lemma \ref{DOP}.
So there exists $L_1=L_1(C_1)>0$ by Lemma
\ref{Uexpdecay} (1) such that
$$
 \mathbf {HSha}(g, n, L)    < C_1/3 \cdot \exp(-\g G
d(o, go)),
$$
for $L>L_1$. Moreover, by Lemma \ref{farhoroball}, there exists $L_2=L_2(C_1)>0$    such that
$$
 \mathbf {ESha}(g, n, L)    < C_1/3 \cdot \exp(-\g G
d(o, go)),
$$
for $L>L_2$. 
Hence,
\begin{equation}\label{F6}
\begin{array}{rl}
\mathbf {ASha}(g,n, \tilde \Delta)  \ge C_1/3 \cdot \exp(-\g G d(o, go)) 
\end{array}
\end{equation}
for $g \in G$ and $n >>0$.

By Lemma \ref{PShadowX},  $
\mu_1(\Pi(ho)) \asymp \exp(-\g G d(o, ho))$ for any $h\in G$. Since $d(o, ho) > d(o, go) + n -\Delta$ for any $h \in
\Omega_{\tilde r,\epsilon, \tilde R}(go, n, \tilde \Delta)$, we obtain
$$
\begin{array}{rl}
\mathbf {ASha}(g, n, \tilde \Delta)  \asymp_{\tilde \Delta}  \sharp \Omega_{\tilde r, \epsilon, \tilde R}(go, n, \tilde\Delta)  \cdot
\exp(-\g G n),
\end{array}
$$
yielding by      (\ref{F6}),  
$$\sharp \Omega_{\tilde r, \epsilon,\tilde R}(go, n, \Delta)  \succ_{\tilde \Delta}  \exp(-\g G n).$$
Thus, $G$ has purely exponential growth in partial cones.

\subsection{\textbf{Direction $(4)\Rightarrow (1)$}:} By Lemma \ref{Uexpdecay} (2), the DOP condition follows from  purely exponential growth for horoballs. 
We consider $g=1$ and so the (\ref{LepsilonDOP}) implies
$$
\begin{array}{rl}
\mathbf {HSha}(1,n, L)\succ      \sum\limits_{j\ge 2L}^n \mathcal{A}_{V}(o_V, j, \Delta)
\end{array}
$$ for any $L>L_0$ and $V\in \mathbb U$. By Lemma \ref{DOP}, the DOP condition is verified, as $n\to \infty$. 

The proof of the theorem is completed.



\bibliographystyle{amsplain}
 \bibliography{bibliography}

\end{document}